\documentclass{amsart}

\usepackage{mathrsfs}
\usepackage{amsfonts}
\usepackage{amsmath}
\usepackage{amssymb}
\usepackage{mathtools}
\usepackage{float}
\usepackage{xcolor}
\usepackage[inline]{enumitem}
\usepackage{csquotes}
\usepackage{subcaption}
\usepackage{tikz}
\usepackage{tikz-cd}
\usetikzlibrary{arrows,automata}
\usepackage{hyperref}
\usepackage{centernot}
\usepackage{varwidth}
\usepackage{bbm}

\usepackage{baskervald}

\newcommand{\pres}[3]{\textnormal{#1} \langle #2 \mid #3 \rangle}

\newcommand{\lra}[1]{\xleftrightarrow{\ast}_{#1}}
\newcommand{\xra}[1]{\xrightarrow{}^\ast_{#1}}
\newcommand{\xr}[1]{\xrightarrow{}_{#1}}

\newcommand{\trev}{\text{rev}}

\newcommand{\CF}{\mathbf{CF}}
\newcommand{\CS}{\mathbf{CS}}

\newcommand{\IND}{\mathbf{IND}}

\newcommand{\ETOL}{\mathbf{ET0L}}
\newcommand{\EDTOL}{\mathbf{EDT0L}}

\newcommand{\cc}{\mathbf{C}}
\newcommand{\cT}{\mathcal{T}}

\newcommand{\cl}{\mathcal{L}}

\newcommand{\bi}{\mathbbm{1}}

\newcommand{\sR}{\mathscr{R}}
\newcommand{\sS}{\mathscr{S}}
\newcommand{\sT}{\mathscr{T}}
\newcommand{\anc}[1]{\nabla^\ast_{#1}}

\DeclareMathOperator{\AFL}{AFL}

\DeclareMathOperator{\Alt}{Alt}

\DeclareFontFamily{U}{wncy}{}
    \DeclareFontShape{U}{wncy}{m}{n}{<->wncyr10}{}
    \DeclareSymbolFont{mcy}{U}{wncy}{m}{n}
    \DeclareMathSymbol{\sh}{\mathord}{mcy}{"78}

\newtheorem{theorem}{Theorem}[section] 

\newtheorem{lemma}[theorem]{Lemma}     
\newtheorem{corollary}[theorem]{Corollary}

\newtheorem{proposition}[theorem]{Proposition}
\newtheorem{theoremB}{Theorem}

\newtheorem{theoremC}{Theorem}

\theoremstyle{definition}
\newtheorem{definition}{Definition}

\newtheorem{example}{Example}

\numberwithin{equation}{section}

\begin{document}

\title[Multiplication tables in free products]{Multiplication tables and word-hyperbolicity in free products of semigroups, monoids, and groups}

\author{Carl-Fredrik Nyberg-Brodda}
\address{Alan Turing Building, Department of Mathematics, University of Manchester, United Kingdom.}
\email{carl-fredrik.nybergbrodda@manchester.ac.uk}

\thanks{The author gratefully acknowledges funding from the Dame Kathleen Ollerenshaw Trust, which is funding his current position as Research Associate at the University of Manchester.}


\date{\today}



\begin{abstract}
This article studies the properties of word-hyperbolic semigroups and monoids, i.e. those having context-free multiplication tables with respect to a regular combing, as defined by Duncan \& Gilman. In particular, the preservation of word-hyperbolicity under taking free products is considered. Under mild conditions on the semigroups involved, satisfied e.g. by monoids or  regular semigroups, we prove that the semigroup free product of two word-hyperbolic semigroups is again word-hyperbolic. Analogously, with a mild condition on the uniqueness of representation for the identity element, satisfied e.g. by groups, we prove that the monoid free product of two word-hyperbolic monoids is word-hyperbolic. The methods are language-theoretically general, and apply equally well to semigroups, monoids, or groups with a $\cc$-multiplication table, where $\cc$ is any reversal-closed super-$\AFL$. In particular, we deduce that the free product of two groups with $\ETOL$ resp. indexed multiplication tables again has an $\ETOL$ resp. indexed multiplication table.
\end{abstract}

\maketitle



\noindent Hyperbolic groups, being groups whose Cayley graphs are hyperbolic metric spaces, were introduced by Gromov in his seminal monograph \cite{Gromov1987}. Subsequently, the theory of hyperbolic groups has grown to one of the most influential areas of group theory. It is natural to wish to generalise hyperbolicity from groups to semigroups and monoids. This can be done in several ways. One way is to consider semigroups and monoids whose (right, undirected) Cayley graphs are hyperbolic as metric spaces. This approach is somewhat brittle: for a trivial example which illustrates this well, if $G$ is \textit{any} group whatsoever, then $G^0$ -- the semigroup obtained from $G$ by adjoining a zero -- has a hyperbolic Cayley graph when this is considered as an undirected graph. A second, and somewhat more robust, approach is to use the methods of formal language theory, which is the starting point of the present article.

Language-theoretic methods in group theory have a rich history spanning the past half century, starting with An\={\i}s\={\i}mov in 1969 \cite{Anisimov1969} and continuing with the seminal \cite{Anisimov1971}. It is in this latter article in which the ``word problem'' of a group, being the formal language of words over a generating set representing the identity element, was introduced and studied. The connections between groups and context-free languages were explored further by An\={\i}s\={\i}mov \cite{Anisimov1972, Anisimov1972b, Anisimov1973, Anisimov1975b, Anisimov1975a}. Muller \& Schupp \cite{Muller1983} (contingent on a weak form of a deep result by Dunwoody \cite{Dunwoody1985}) subsequently proved a striking classification: a finitely generated group has context-free word problem if and only if it is virtually free. This decisively demonstrated the depth of the connection first uncovered by An\={\i}s\={\i}mov.

With the importance of context-free languages to the theory of groups, it seems natural to desire a purely language-theoretic definition of hyperbolicity in groups. One such definition was given by Grunschlag \cite{Grunschlag1999}, who proved that a group is hyperbolic if and only if its word problem is generated by a terminating growing context-sensitive grammar. An arguably more elegant characterisation, using the weaker expressive power of context-free languages, was given by Gilman \cite{Gilman2002}. This definition has the added benefit of being generalisable directly to semigroups, which was done by Duncan \& Gilman \cite{Duncan2004}. To distinguish from the geometric variant, this form of hyperbolicity is called \textit{word-hyperbolicity}. Loosely speaking, a semigroup $S$ is word-hyperbolic if there exists a regular language of representatives (with no requirement of uniqueness) such that the multiplication table for $S$ with respect to this language can be described by a context-free language. This definition (which is described formally in \S\ref{Subsec:Intro-hyp}) is equivalent to geometric hyperbolicity for groups \cite[Corollary~4.3]{Duncan2004} and for completely simple semigroups \cite[Theorem~4.1]{Fountain2004}. In general, however, word-hyperbolic semigroups are a more restricted class than hyperbolic semigroups, and appear somewhat more amenable to general results than the geometric approach to hyperbolic semigroups.\footnote{Having said this -- and as a paper on hyperbolicity of semigroups may be considered to be skewed without some references to the geometry of semigroups -- there is a recent trend, pioneered by Gray \& Kambites, in successfully handling the directed geometry of semigroups in a manner extending the usual geometric group theory (e.g. the Milnor-Schwarz lemma \cite{Schwarz1955}), see \cite{Gray2011, Gray2013, Gray2014, Gray2017, Gray2020}; cf. also \cite{Silva2004, Hoffmann2006, Kuske2006, Awang2017, Garreta2021}.} For example, just as in groups, there are links between word-hyperbolicity and automaticity in semigroups. 

Whenever a generalisation (e.g. of hyperbolicity in groups) is made, it is useful to ask: what properties should be desired to be retained, and which should not? For example, hyperbolicity in groups is independent of generating set chosen, which is a rather (one may argue) essential and desirable property; this property holds also for word-hyperbolic semigroups \cite[Theorem~3.4]{Duncan2004}. Furthermore, the word problem is well-known to be decidable in all hyperbolic groups (in linear time); for word-hyperbolic semigroups, the word problem is also decidable \cite[Theorem~3.8]{Hoffmann2002}, in fact in polynomial time \cite[Theorem~7.1]{Cain2016}. 

On the other hand, while hyperbolic groups are \textit{automatic} \cite[Theorem~3.4.5]{Epstein1992}, it is not true that word-hyperbolic semigroups are always automatic \cite[Example~7.7]{Hoffmann2002}. Similarly, while the isomorphism problem is decidable for hyperbolic groups \cite{Dahmani2011}, it is undecidable in general for word-hyperbolic semigroups \cite[Theorem~4.3]{Cain2016}. When considering which properties are desired -- and are reasonable to desire -- if a definition is found to not satisfy one such desired property, then an amendment to the definition which forces this property to hold may be considered.\footnote{In fact, two amendments have already been proposed, by Hoffmann \& Thomas \cite{Hoffmann2010} (which recovers automaticity and having word problem decidable in $O(n \log n$-time) and Cain \& Pfeiffer \cite{Cain2016}, respectively.} It is the view of the author that free products of word-hyperbolic semigroups ought to be word-hyperbolic; free products are free constructions, and free objects (e.g. free semigroups) are word-hyperbolic. While we are not able to prove this result with exactly this statement, the main results of this article will demonstrate that any possible counterexample to the general statement will be exceptional, rather than the norm. 

The outline of the paper is as follows. In \S\ref{Sec:Notation_et_al} we will give some background, necessary definitions, and notation. In particular, in \S\ref{Subsec:ETOL-subst} we will give a brief overview of the connections between substitutions in formal language theory (crucial to the arguments in subsequent sections) and $\ETOL$ systems. In \S\ref{Sec:SGPFP}, we shall prove the main result for semigroup free products: 

\begin{theoremB}\label{Thm:Sgp_extendable_result}
Let $S_1, S_2$ be $1$-extendable word-hyperbolic semigroups. Then the free product $S_1 \ast S_2$ is word-hyperbolic. 
\end{theoremB}

The technical condition for a word-hyperbolic semigroups $S$ to be $1$-extendable is defined and explored in \S\ref{Subsec:Intro-hyp}, and loosely speaking consists of a condition ensuring that the word-hyperbolic structure for $S$ does not collapse if one adjoins an identity to $S$. In particular (see Lemma~\ref{Lem:Types-of-extendable}), any monoid and any (von Neumann) regular semigroup is $1$-extendable, so Theorem~\ref{Thm:Sgp_extendable_result} applies when $S_1$ and $S_2$ are from either of these classes.

To deal with monoid free products (and, as a particular case, group free products), in \S\ref{Sec:Polypartisan} we will first develop some technical purely language-theoretic tools which we call \textit{polypartisan ancestors}. Loosely speaking, polypartisan ancestors model a form of sequential rewriting with respect to rules of the form $a \to W$, where $a$ is a single letter, and where the word $w$ to which the rewriting is applied is divided into some fixed number $k$ parts such that each part is rewritten using possibly different sets of rules. In \S\ref{Sec:MFP}, we will first use our previous results on semigroup free products to prove the following main result: 

\begin{theoremB}\label{Thm:Monoid-fp-theorem-1-unqiueness}
Let $M_1, M_2$ be two word-hyperbolic monoids with $1$-uniqueness (with uniqueness). Then the monoid free product $M_1 \ast M_2$ is word-hyperbolic with $1$-uniqueness (with uniqueness).
\end{theoremB}

Here, $1$-uniqueness (which is defined in \S\ref{Subsec:1-uniqueness}) is the condition on a word-hyperbolic monoid $M$, with regular language of representatives $R$, saying: the only word in $R$ which represents the identity element is the empty word. In particular, one can show that hyperbolic groups are word-hyperbolic with $1$-uniqueness, and as a corollary we derive, by language-theoretic means, the well-known group-theoretic result that the free product of hyperbolic groups is hyperbolic. Using polypartisan ancestors, we also deduce (Theorem~\ref{Thm:star-hyp-fp}) that the monoid free product of $\star$-word-hyperbolic monoids is $\star$-word-hyperbolic. Here, $\star$-word-hyperbolicity refers to a condition which can be seen, in a sense made precise in \S\ref{Subsec:NewDef}), as ``complementary'' to $1$-uniqueness. 

Finally, in \S\ref{Sec:Finalremarks}, we use the language-theoretic generality in which the article is written to note that our results generalise immediately from context-free multiplication tables to much more general situations, including $\ETOL$ and indexed multiplication tables, leading to Theorem~\ref{Thm:A'} and Theorem~\ref{Thm:B'}, being more general versions of the main results of this article. In particular, we deduce (Corollary~\ref{Cor:ETOLtabledgroups}) that the free product of two groups, each having an $\ETOL$ multiplication table, again has an $\ETOL$ multiplication table; the analogous statement (Corollary~\ref{Cor:INDtabledgroups}) for groups with indexed multiplication tables is also deduced.

\setcounter{theoremB}{0}

\clearpage

\section{Introduction and Notation}\label{Sec:Notation_et_al}

The paper also assumes familiarity with the basics of the theory of semigroup, monoid, and group presentations, which will be written as $\pres{Sgp}{A}{\sR}$, $\pres{Mon}{A}{\sR}$, and $\pres{Gp}{A}{\sR}$, respectively. For further background see e.g. \cite{Adian1966,Magnus1966,Neumann1967, Lyndon1977,Campbell1995}.

\subsection{Formal languages}

\noindent We assume the reader is familiar with the fundamentals of formal language theory. In particular, a full $\AFL$ (\textit{abstract family of languages}) is a class of languages closed under homomorphism, inverse homomorphism, intersection with regular languages, union, concatenation, and the Kleene star. Furthermore, a class $\cc$ is \textit{reversal-closed} if for all $L \in \cc$, we have $L^\trev \in \cc$. Here $L^\trev$ denotes the language of all words in $L$ read backwards (see \S\ref{Subsec:Rewritingsystems} for a formal definition). For some background on this, and other topics in formal language theory, we refer the reader to standard books on the subject \cite{Harrison1978, Berstel1979, Hopcroft1979, Rozenberg1997}. The class of context-free languages will be denoted $\CF$.

We will also, in \S\ref{Sec:Finalremarks} and \S\ref{Subsec:ETOL-subst} make reference to the class $\IND$ of indexed languages. The latter was introduced in Aho's Ph.D. thesis \cite{Aho1967}, see also \cite{Aho1968} as an extension of the context-free languages; we refer the reader to e.g. \cite{Hayashi1973, Gilman1996, Smith2017} or \cite[Chapter~14]{Hopcroft1979} for particularly readable definitions. Finally, we shall make some reference to the classes $\ETOL$ and $\EDTOL$ in \S\ref{Sec:Finalremarks} (but not in the main sections of the paper). These are examples of $\mathbf{L}$-languages, which arise from $\mathbf{L}$-systems. The theory of $\mathbf{L}$-systems originated in 1968 in the work of Lindenmayer \cite{Lindenmayer1968a, Lindenmayer1968b} (whence the $\mathbf{L}$) as a theory for the parallel branching of filamentous organisms in biology, but subsequently grew into a core branch of formal language theory \cite{Rozenberg1974, Herman1975, Rozenberg1976, Rozenberg1977}. Because of this vast literature (and as we shall not need the definitions), we shall not define either $\ETOL$ or $\EDTOL$, instead referring the reader to more recent articles on the subject (e.g. especially \cite{Ciobanu2018}, see also \cite{Istrate1997, Rabkin2012, Brough2016}). The research topic remains very active; particularly the connections between $\ETOL$ and $\EDTOL$ languages and equations over groups and monoids have flourished in recent years \cite{Ferte2014, Ciobanu2016, Ciobanu2018, Ciobanu2019, Diekert2020, Ciobanu2021, Levine2021, Evetts2022, Levine2022, Jez2022}. There are also recent links with geometric group theory. For example, Bridson \& Gilman \cite{Bridson1996} famously proved that any $3$-manifold group admits a combing which is an indexed language; in fact, their combing is an $\ETOL$ language \cite{Ciobanu2018} (note that $\CF \subsetneq \ETOL \subsetneq \IND$). 

A useful analogy to keep in mind is the following: the class $\CF$ models closure under sequential recursion; the class $\ETOL$ models closure under parallel recursion. See \S\ref{Subsec:ETOL-subst} for further details on this analogy, particularly Theorem~\ref{Thm:CF_ETOL_smallest_AFLs}, as well as \cite{Rozenberg1979}. Finally, $\CF, \ETOL$, and $\IND$ are all easily seen to be reversal-closed.

\subsection{Rewriting systems}\label{Subsec:Rewritingsystems}

Let $A$ be a finite alphabet, and let $A^\ast$ denote the free monoid on $A$, with identity element denoted $\varepsilon$ or $1$, depending on the context. Let $A^+$ denote the free semigroup on $A$, i.e. $A^+ = A^\ast - \{ \varepsilon\}$. For $u, v \in A^\ast$, by $u \equiv v$ we mean that $u$ and $v$ are the same word. For $w \in A^\ast$, we let $|w|$ denote the \textit{length} of $w$, i.e. the number of letters in $w$. We have $|\varepsilon| = 0$. If $w \equiv a_1 a_2 \cdots a_n$ for $a_i \in A$, then we let $w^\trev$ denote the \textit{reverse} of $w$, i.e. the word $a_n a_{n-1} \cdots a_1$. Note that $^\trev \colon A^\ast \to A^\ast$ is an anti-homomorphism, i.e. $(uv)^\trev \equiv v^\trev u^\trev$ for all $u, v \in A^\ast$. If $X \subseteq A^\ast$, then we let $X^\trev = \{ x^\trev \mid x \in X \}$. If the words $u, v \in A^\ast$ are equal in the monoid $M = \pres{Mon}{A}{\sR}$, then we denote this $u =_M v$. By $M^\trev$ we mean the \textit{reversed} monoid
\[
M^\trev = \pres{Mon}{A}{ \{ u^\trev = v^\trev \mid (u,v) \in \sR \}}.
\]
For $w_1, w_2 \in A^\ast$,  $w_1 =_M w_2$ if and only if $w_1^\trev =_{M^\trev} w_2^\trev$. That is, $M$ and $M^\trev$ are anti-isomorphic (if $G$ is a group, then clearly $G \cong G^\trev$). Finally, when we say that a monoid $M$ is generated by a set $A$, we mean that there exists a surjective homomorphism $\pi \colon A^\ast \to M$. We use analogous terminology for semigroups and groups.

We give some notation for rewriting systems. For an in-depth treatment and further explanations of the terminology, see e.g. \cite{Book1982, Jantzen1988, Book1993}. A \textit{rewriting system} $\sR$ on $A$ is a subset of $A^\ast \times A^\ast$. We shall denote rewriting systems by script letters, e.g. $\sR, \sS, \sT$. An element of $\sR$ is called a \textit{rule}. The system $\sR$ induces several relations on $A^\ast$. We will write $u \xr{\sR} v$ if there exist $x, y \in A^\ast$ and a rule $(\ell, r) \in \sR$ such that $u \equiv x\ell y$ and $v \equiv xry$. We let $\xra{\sR}$ denote the reflexive and transitive closure of $\xr{\sR}$. We denote by $\lra{\sR}$ the symmetric, reflexive, and transitive closure of $\xr{\sR}$. The relation $\lra{\sR}$ defines the least congruence on $A^\ast$ containing $\sR$. For $X \subseteq A^\ast$, we let $\anc{\sR}(X)$ denote the set of \textit{ancestors} of $X$ with respect to $\sR$, i.e. $\anc{\sR}(X) = \{ w \in A^\ast \mid \exists x \in X \textnormal{ such that } w \xra{\sR} x \}$. The monoid $\pres{Mon}{A}{\sR}$ is identified with the quotient $A^\ast / \lra{\sR}$. For a rewriting system $\sT \subseteq A^\ast \times A^\ast$ and a monoid $M = \pres{Mon}{A}{\sR}$, we say that $\sT$ is $M$\textit{-equivariant} if for every rule $(u, v) \in \sT$, we have $u =_M v$. That is, $\sT$ is $M$-equivariant if and only if $\lra{\sT} \subseteq \lra{\sR}$.

 Let $u, v \in A^\ast$ and let $n \geq 0$. If there exist words $u_0, u_1, \dots, u_n \in A^\ast$ such that 
\[
u \equiv u_0 \xr{\sR} u_1 \xr{\sR} \cdots \xr{\sR} u_{n-1} \xr{\sR} u_n \equiv v,
\]
then we denote this $u \xr{\sR}^n v$, i.e. $u$ rewrites to $v$ in $n$ steps. Thus $\xra{\sR} = \bigcup_{n \geq 0} \xr{\sR}^n$.

A rewriting system $\sR \subseteq A^\ast \times A^\ast$ is said to be \textit{monadic} if $(u, v) \in \sR$ implies $|u| \geq |v|$ and $v \in A \cup \{ \varepsilon \}$. We say that $\sR$ is \textit{special} if $(u, v) \in \sR$ implies $v \equiv \varepsilon$. Every special system is monadic. Let $\cc$ be a class of languages. A monadic rewriting system $\sR$ is said to be $\cc$ if for every $a \in A \cup \{ \varepsilon \}$, the language $\{ u \mid (u, a) \in \sR \}$ is in $\cc$. Thus, we may speak of e.g. $\cc$-monadic rewriting systems or context-free monadic rewriting systems. Monadic rewriting systems are extensively treated in \cite{Book1982}.

\begin{definition}
Let $\cc$ be a class of languages. Let $\sR \subseteq A^\ast \times A^\ast$ be a rewriting system. Then we say that $\sR$ is $\cc$\textit{-ancestry preserving} if for every $L \subseteq A^\ast$ with $L \in \cc$, we have $\anc{\sR}(L) \in \cc$. If every $\cc$-monadic rewriting system is $\cc$-ancestry preserving, then we say that $\cc$ has the \textit{monadic ancestor property}.
\end{definition}

The terminology \textit{monadic ancestor property} was introduced by the author in \cite{NybergBrodda2020b}, and also appears in \cite{NybergBrodda2020c, NybergBrodda2021f}, but was treated implicitly already in \cite{Book1982, Jantzen1988}, see especially \cite[Lemma~3.4]{Jantzen1988}. The idea of defining classes of languages via ancestry in rewriting systems is not new, and can be traced back at least to e.g. McNaughton et al's Church-Rosser languages \cite{McNaughton1988} or Beaudry et al.'s McNaughton languages \cite{Beaudry2002, Beaudry2003}.

\begin{example}\label{Ex:CF-is-super-afl}
If $\sR \subseteq A^\ast \times A^\ast$ is a context-free monadic rewriting system, and $L \subseteq A^\ast$ is a context-free language, then $\anc{\sR}(L)$ is a context-free language \cite[Theorem~2.2]{Book1982}. That is, every $\CF$-monadic rewriting system is $\CF$-ancestry preserving. Hence the class of context-free languages has the monadic ancestor property. 
\end{example}

Having the monadic ancestor property is analogous to being closed under sequential recursion; see \S\ref{Subsec:ETOL-subst} for further elaboration on this. This gives rise to the notion of a super-$\AFL$. 

\begin{definition}\label{Def:super-AFL}
Let $\cc$ be an $\AFL$. Then $\cc$ is said to be a \textit{super-$\AFL$} if it has the monadic ancestor property.
\end{definition}

Hence, by Example~\ref{Ex:CF-is-super-afl}, $\CF$ is a super-$\AFL$. For the main body of the text, this is the only super-$\AFL$ we will deal with; see, however, \S\ref{Subsec:ETOL-subst} for a broader discussion, and \S\ref{Sec:Finalremarks} for generalisations of our results to all reversal-closed super-$\AFL$s. The primary reason for dealing only with context-free languages comes from the importance of $\CF$ with regards to word-hyperbolicity.

\subsection{Word-hyperbolicity}\label{Subsec:Intro-hyp}

Let $S$ be a semigroup, finitely generated by some set $A$, with associated surjective homomorphism $\pi_S \colon A^+ \to S$. Let $R \subseteq A^+$ be a regular language. If $\pi_S(R) = S$, i.e. every element of $S$ is represented by some word from $R$, then we say that $R$ is a \textit{regular combing} of $S$. If $\pi_S$ is bijective when restricted to $R$, then we say that $R$ is a regular combing \textit{with uniqueness}. Let $\#_1, \#_2$ be two new symbols, and let 
\begin{equation}\label{Def:MULT_TABLE}
\cT_S(R) = \{ u \#_1 v \#_2 w^\trev \mid u, v, w \in R \textnormal{ such that } u \cdot v =_S w \}.
\end{equation}
We say that $\cT_R(S)$ is a \textit{multiplication table} for $S$ (with respect to $R$). If this table is context-free, i.e. if $\cT_S(R) \in \CF$, then we say that $S$ is a \textit{word-hyperbolic semigroup} (with respect to the combing $R$). If $R$ is additionally a combing with uniqueness, then we say that $\cT_S(R)$ is word-hyperbolic \textit{with uniqueness}. Not every word-hyperbolic semigroup is word-hyperbolic with uniqueness \cite{Cain2012}. 

The above notion of hyperbolicity was introduced by Duncan \& Gilman \cite{Duncan2004}. One can show that if $S$ is hyperbolic with respect to one choice of finite generating set, then it is hyperbolic with respect to every such choice \cite[Theorem~3.4]{Duncan2004}. However, note that even if $\cT_{S}(R_1) \in \CF$ for some regular combing $R_1$, there may still be some regular combing $R_2$ of $S$ such that $\cT_S(R_2) \not\in \CF$. For extensions of the condition $\cT_S(R) \in \CF$ to e.g. $\cT_S(R) \in \ETOL$ or $\cT_S(R) \in \IND$, see \S\ref{Sec:Finalremarks}. 

We extend this definition in the obvious way to monoids (and groups), by substituting $A^\ast$ for $A^+$. Thus a monoid $M$ generated by $A$ is word-hyperbolic ``as a monoid'' if and only if there exists a regular combing $R \subseteq A^\ast$ such that $\cT_M(R) \in \CF$. However, by \cite[Theorem~3.5]{Duncan2004}, a monoid is word-hyperbolic ``as a monoid'' if and only if it is word-hyperbolic as a semigroup (in the above sense). We shall therefore speak of ``word-hyperbolic monoids'' always referring to a regular combing $R \subseteq A^\ast$. In fact, it is not difficult to see, by using a rational transduction, that if $M$ is word-hyperbolic with respect to a combing $R \subseteq A^+$, then it is word-hyperbolic with respect to $R \cup \{ \varepsilon \}$ (see e.g. the first paragraph in the proof of Lemma~\ref{Lem:right-stabiliser-extends}). We may thus assume without loss of generality assume that any regular combing for $M$ includes the empty word (which necessarily represents the identity element). If $M$ is word-hyperbolic with respect to the combing $R$, and the only word in $R$ representing the identity element of $M$ is the empty word, then we say that $M$ is word-hyperbolic \textit{with $1$-uniqueness}.

One can show that a group is word-hyperbolic if and only if it is hyperbolic in the usual sense, i.e. the sense of Gromov \cite[Theorem~4.3]{Duncan2004}. Furthermore, one can show that, due to the Muller-Schupp theorem, if $G$ is a group generated by $A$, then $\cT_{G}(A^\ast)$ is context-free if and only if $G$ is virtually free \cite[Theorem~2(2)]{Gilman2002}, a condition which is significantly stronger than hyperbolicity. Indeed, more generally, it is not difficult to see that a semigroup $S$ is word-hyperbolic with respect to the combing $A^+$ if and only if $S$ has context-free word problem (in the sense of Duncan \& Gilman \cite[\S5]{Duncan2004}). 

For brevity, for $i=1,2$ we let $A_{\#_i} = A \cup \{ \#_i \}$, and let $A_\# = A \cup \{ \#_1, \#_2 \}$. 

\subsection{1-extendability}

We now define a slightly technical condition, which will prove useful in \S\ref{Sec:SGPFP}. Let $S$ be a semigroup. We define $S^\bi$ to be the semigroup with an identity $\bi$ adjoined, \textit{regardless} of whether $S$ has an identity element already or not.\footnote{If $S$ is a monoid, then defining $S^\bi$ in this manner (rather than simply taking $S^\bi = S$) is only a technicality, but is used to avoid some other language-theoretic technicalities.}

\begin{definition}[$1$-extendable]\label{Def:1-extendable}
Let $S$ be a word-hyperbolic semigroup with respect to a regular combing $R \subseteq A^+$. We say that $S$ is $1$\textit{-extendable} if $S^\bi$ is word-hyperbolic with respect to the regular combing $R \cup \{ \varepsilon \}$. 
\end{definition}

Thus if $S$ is a $1$-extendable word-hyperbolic semigroup, then $S^\bi$ is word-hyperbolic. We do not know if the converse holds in general. Our main interest in $1$-extendability will be in the statement of Theorem~\ref{Thm:Sgp_extendable_result}, in which we shall show that the free product of $1$-extendable word-hyperbolic semigroups is again word-hyperbolic. We begin by showing that $1$-extendability is not particularly elusive. 

\begin{lemma}[Kambites]\label{Lem:right-stabiliser-extends}
Let $S$ be a word-hyperbolic semigroup. If every element of $S$ has a right stabiliser (i.e. for every $s \in S$ there exists some $t \in S$ with $st = s$), then $S$ is $1$-extendable.
\end{lemma}
\begin{proof}
Suppose $S$ is generated by the finite set $A$, with $R \subseteq A^+$ a regular combing and $\cT_S(R)$ context-free. As noted by Duncan \& Gilman \cite[Question~1]{Duncan2004}, to show that $S^\bi$ is word-hyperbolic it suffices to show that the language $Q = \{ u \# v^\trev \mid u, v \in R, u=_S v \}$ is context-free, as 
\begin{align*}
\cT_{S^\bi}(R \cup \{ \varepsilon \}) = \{ (\#_1 \#_2)\} &\cup \{ u \#_1 \#_2 v^\trev \mid u, v \in R, u =_S v \} \cup \\ &\cup \{ \#_1 u \#_2 v^\trev \mid u, v \in R, u =_S v \} \cup \cT_S(R),
\end{align*} 
i.e. it is a union of $\cT_S(R)$ and languages obtainable from $Q$ by a rational transduction, and hence also context-free.

For every $a \in A$, let $a' \in R$ be a word such that $aa' =_S a$, i.e. a right stabiliser for $a$, which exists by assumption. Let $A' = \{ a' \mid a \in A \}$. Then for every $u \in R$, say $u \equiv a_1 a_2 \cdots a_n$, we have that $u a_n' =_S u$. By partitioning $R$ based on the final letters of words (which is well-defined as $\varepsilon \not\in R$), we find that the language
\[
U = \bigcup_{a \in A} \left( R / \{ a \} \right) a \#_1 a'
\]
is a regular language, being a finite union of (pairwise disjoint) regular languages. Now
\begin{align*}
\cT_S(R) \cap (U \#_2 R^\trev) &= \{ u \#_1 v \#_2 w^\trev \mid u, v, w \in R, uv =_S w, \exists a \in A \colon u \in A^\ast a, v \equiv a' \} \\
&= \{ u \#_1 a' \#_2 w^\trev \mid u, w \in R, \exists a \in A \colon u \in A^\ast a, u a' =_S w, u a' =_S u \} \\
&= \{ u \#_1 a' \#_2 w^\trev \mid u, w \in R, \exists a \in A \colon u \in A^\ast a, u =_S w \} =: L.
\end{align*}
This latter language $L$ is just given by 
\begin{equation}\label{Eq:Union_aa}
L = \bigcup_{a \in A} \{ u \#_1 a' \#_2 w^\trev \mid u, w \in R, u \in A^\ast a, u =_S w \} =: \bigcup_{a \in A} L_a.
\end{equation}
Now for every $a \in A$, we have $L_a = L \cap A^\ast \#_1 a' \#_2 A^\ast$. Hence, as $\CF$ is closed under union and intersection with regular languages, we have that $L \in \CF$ if and only if for all $a \in A$, $L_a \in \CF$. As $\cT_S(R) \in \CF$, and $U \#_2 R^\trev$ is regular, we have $L \in \CF$, and thus also $L_a \in \CF$ for all $a \in A$. 

For every $a \in A$, let $\varrho_a$ be the rational transduction of $L_a$ defined by deleting $\#_1 a' \#_2$ in the input word and replacing it by $\#$ in the output word, and fixing all other parts of input word in $L_a$. Then
\[
L_a' := \varrho_a(L_a) = \{ u \# w^\trev \mid u, w \in R, u \in A^\ast a, u =_S w \}.
\]
As $\CF$ is closed under rational transduction, we have $L_a' \in \CF$ for all $a \in A$. But clearly
\[
\bigcup_{a \in A} L_a' = \{ u \# w^\trev \mid u, w \in R, u =_S w \} = Q.
\]
As $\CF$ is closed under finite unions, we have $Q \in \CF$, as was to be shown. 
\end{proof}

The author thanks Mark Kambites for suggesting Lemma~\ref{Lem:right-stabiliser-extends} and its proof. We rephrase the above result to our current situation of $1$-extendability, and note the following direct consequence:

\begin{lemma}\label{Lem:Types-of-extendable}
Let $S$ be a word-hyperbolic semigroup. If $S$ is either:
\begin{enumerate}
\item a monoid;
\item a (von Neumann) regular semigroup; or
\item a word-hyperbolic semigroup with uniqueness,
\end{enumerate}
then $S$ is $1$-extendable. 
\end{lemma}
\begin{proof}
Parts (1) and (2) are direct corollaries of Lemma~\ref{Lem:right-stabiliser-extends}. Part (3) is immediate, and already noted by Duncan \& Gilman \cite[Question~1]{Duncan2004}.
\end{proof}

In particular, we find that hyperbolic groups are $1$-extendable. On a philosophical note, we remark (for reasons not too dissimilar from those which shall be discussed later in \S\ref{Subsec:NewDef}) that $1$-extendability strikes the author as a very natural condition for working with word-hyperbolic semigroups in the first place. 

\subsection{Semigroup free products}

Semigroup free products can be found described in e.g. \cite[Ch.~8.2]{Howie1995}. We give an overview of this theory here, with some additional terminology that will simplify later notation. 

Let $S_1, S_2$ be semigroups defined by $S_i = \pres{Sgp}{A_i}{\sR_i}$ for $i=1,2$, assuming without loss of generality that $A_1 \cap A_2 = \varnothing$. The \textit{semigroup free product} $S = S_1 \ast S_2$ is defined as $S = \pres{Sgp}{A_1 \cup A_2}{\sR_1 \cup \sR_2}$. We will identify $S$ with the semigroup whose elements are all finite non-empty alternating sequences $(s_1, s_2, \dots, s_n)$ of elements $s_i \in S_1 \cup S_2$, where \textit{alternating} means that $s_i$ and $s_{i+1}$ come from different factors for $1 \leq i < n$. We write $s_i \sim s_{j}$ if $s_i$ and $s_{j}$ come from the same factor, and $s_i \not\sim s_j$ otherwise. We always have $s_i \sim s_i$ and $s_i \not\sim s_{i+1}$. Given any non-empty sequence $s = (s_1, s_2, \dots, s_n)$ with $s_i \in S_1 \cup S_2$, we define the \textit{alternatisation} $s'$ of $s$ to simply be $s' = s$ if $s$ is alternating, and otherwise -- if, say, $s_i \sim s_{i+1}$ -- we define $s'$ as the alternatisation of $(s_1, \dots, s_i \cdot s_{i+1}, \dots, s_n)$. Clearly, the alternatisation of $s$ is a uniquely defined alternating sequence.

The product of two alternating sequences in $S$ is given by the alternatisation of the concatenation of the sequences. That is, explicitly, multiplication in $S$ is given by
\begin{equation}\label{Eq:sgp_fp_multiplication}
(s_1, \dots, s_n) \cdot (t_1, \dots, t_m) = \begin{cases}
			(s_1, \dots, s_n, t_1, \dots, t_m), & \text{if $s_n \not\sim t_1$,}\\
            (s_1,\dots, s_n t_1, \dots, t_m), & \text{otherwise.}
		 \end{cases}
\end{equation}
See e.g. \cite[Eq. 8.2.1]{Howie1995}. Note that, in particular, the semigroup free product of two monoids is never a monoid. We will now define monoid free products in a similar manner. 

\subsection{Monoid free products}\label{Subsec:Intro_MFP}

Let $M_1, M_2$ be monoids defined by $M_i = \pres{Mon}{A_i}{\sR_i}$ for $i=1,2$, assuming without loss of generality that $A_1 \cap A_2 = \varnothing$. The \textit{monoid free product} $M = M_1 \ast M_2$ is defined as $M = \pres{Mon}{A_1 \cup A_2}{\sR_1 \cup \sR_2}$. We will identify $M$ with the monoid whose elements are all finite reduced alternating sequences $(m_1, m_2, \dots, m_n)$ of elements $m_i \in M_1 \cup M_2$, where \textit{reduced} means that $m_i \neq 1$ for all $1 \leq i \leq n$. Given an alternating sequence $s = (s_1, s_2, \dots, s_n)$, we define the \textit{reduction} $s'$ of $s$ to be $s$ if $s$ is already reduced; and as the reduction of the alternatisation of the subsequence $(s_{i_1}, s_{i_2}, \dots, s_{i_k})$ consisting of precisely those $s_{i_j}$ which satisfy $s_{i_j} \neq 1$. Clearly, the reduction $s'$ of $s$ is a uniquely defined reduced alternating sequence. 

The product of two reduced sequences in $M$ is then defined as the reduction of the concatenation of the sequences. Hence, similar to \eqref{Eq:sgp_fp_multiplication}, we easily find an explicit expression for multiplication of elements in a monoid free product as:
\begin{equation}\label{Eq:mon_fp_multiplication}
(s_1, \dots, s_n) \cdot (t_1, \dots, t_m) = \begin{cases}
			(s_1, \dots, s_n, t_1, \dots, t_m), & \text{if $s_n \not\sim t_1$,}\\
            (s_1, \dots, s_n t_1, \dots, t_m), & \text{if $s_nt_1 \neq 1$,} \\
            (s_1, \dots, s_{n-1}) \cdot (t_2, \dots, t_m), & \text{if $s_nt_1 = 1$.}
		 \end{cases}
\end{equation}

See e.g. \cite[p. 266]{Howie1995}. Unlike the case of the semigroup free product, the empty sequence is always an identity element for $M$, so the free product of two monoids is always (obviously) a monoid. We also remark on the recursive definition of multiplication in the third case of \eqref{Eq:mon_fp_multiplication}. We may, of course, have that $s_{n-1} \cdot t_2 = 1$, in which case we continue reducing. In particular, the monoid free product of two groups is a group, and hence the monoid free product of two groups coincides with the usual group free product of the same groups.

\subsection{Alternating words and combings}\label{Subsec:Alternatingwords}

We will make the following definition of alternating words, which is useful in describing the language theory of free products. Let $R_1, R_2$ be regular languages over some alphabets $A_1, A_2$ respectively, with $A_1 \cap A_2 = \varnothing$ (and hence $R_1 \cap R_2 = \varnothing$ or $R_1 \cap R_2 = \{ \varepsilon \}$). Let $w \in (R_1 \cup R_2)^+$ be a non-empty word. Then we can factorise $w$ -- not necessarily uniquely! -- as a product $w \equiv x_1 x_2 \cdots x_n$, where for every $1 \leq i \leq n$ we have $x_i \in R_1 \cup R_2$ and $x_i \not\equiv \varepsilon$. Any such factorisation $x_1 x_2 \cdots x_n$ of $w$ gives rise to a parametrisation $X \colon \mathbb{N} \to \{ 1, 2\}$ uniquely defined on $\{ 1, \dots, n \}$ (as $A_1 \cap A_2 = \varnothing$) by $X(i) = j$ when $x_i \in X_j$. If $X(i) = X(i+1)$ for some $i$, then we will write this as $x_i \sim x_{i+1}$ (context will always make this slightly abusive notation clear). If $X$ is such that $x_i \not\sim x_{i+1}$, i.e. $X(i) \neq X(i+1)$, for all $1 \leq i < n$, which is to say that $X$ is a standard parametrisation when restricted to $\{ 1, \dots, n\}$, then we say that the factorisation $x_1 x_2 \cdots x_n$ of $w$ is \textit{alternating}. In this case, we may without loss of generality assume $X$ is a standard parametrisation. 

If $w$ admits an alternating factorisation, then we say that $w$ is an $(R_1, R_2)$-\textit{alternating word} (or simply \textit{alternating word}, if context makes the regular languages $R_1, R_2$ clear). It is clear that $w$ admits at most one alternating factorisation, and hence, if $w$ is an alternating word, then we may speak of \textit{the} alternating factorisation of $w$, with associated standard parametrisation $X$. We will for convenience always also say that the empty word is alternating, with the ``unique'' alternating factorisation $\varepsilon$ (if $\varepsilon \in R_1 \cap R_2$, then we simply for convenience choose $\varepsilon \in R_1$). Note that not every factorisation as a word over $(R_1 \cup R_2)^+$ of an alternating word is alternating. 

The language of all $(R_1, R_2)$-alternating words is regular, being the language
\begin{equation}\label{Eq:Alt_words_Def}
(R_1 R_2)^\ast \cup (R_2 R_1)^\ast \cup (R_1 R_2)^\ast R_1 \cup (R_2 R_1)^\ast R_2. 
\end{equation}

We will denote the language \eqref{Eq:Alt_words_Def} as $\Alt(R_1, R_2)$. We denote by $\Alt^+(R_1, R_2)$ the language $\Alt(R_1, R_2) - \{ \varepsilon \}$ of non-empty alternating words. 

\begin{lemma}\label{Lem:SGPFP-alt-is-combing}
Let $S_1, S_2$ be two semigroups, finitely generated by disjoint sets $A_1$ resp. $A_2$ and with regular combings $R_1$ resp. $R_2$. Then the language $\Alt^+(R_1, R_2)$ is a regular combing of the semigroup free product $S = S_1 \ast S_2$.
\end{lemma}
\begin{proof}
Let $(s_1, s_2, \dots, s_k) \in S$ be an alternating sequence, with associated parametrisation $X$, i.e. so that $s_i \in S_{X(i)}$ for all $1 \leq i \leq k$. For every $1 \leq i \leq k$, there is some $r_i \in R_{X(i)}$ such that $\pi_{X(i)}(r_i) = s_i$, as $R_{X(i)}$ is a combing of $S_{X(i)}$. Hence $\pi(r_1 r_2 \cdots r_k) = (s_1, s_2, \dots, s_k)$, and as $r_1 r_2 \cdots r_k \in \Alt^+(R_1, R_2)$, we have the result. 
\end{proof}

Now the following follows immediately from Lemma~\ref{Lem:SGPFP-alt-is-combing} and standard normal form lemmas for semigroup free products. 

\begin{lemma}\label{Lem:SGPFP-combing-equality}
Let $S_1$ and $S_2$ be as in Lemma~\ref{Lem:SGPFP-alt-is-combing}. Let $S= S_1 \ast S_2$ denote their semigroup free product, and let $u, v \in R$ such that 
\[
u \equiv u_1 u_2 \cdots u_n, \quad v \equiv v_1 v_2 \cdots v_k
\]
are the unique alternating factorisations of $u$ and $v$, respectively, and with associated standard parametrisations $X$ resp. $Y$. Then $u =_S v$ if and only if 
\begin{enumerate}
\item $n = k$ and $X=Y$;
\item $u_i =_{S_{X(i)}} v_i$ for all $1 \leq i \leq n$.
\end{enumerate}
\end{lemma}

Finally, we give an explicit expression for how multiplication works in semigroup free products with respect to the combing $R$. 

\begin{lemma}\label{Lem:Mult_table_SGP_lemma}
Let $x \equiv x_1 x_2 \cdots x_n \in R$ be an alternating product such that $x_i \in R_{X(i)}$ for some standard parametrisation $X$. Let $w_1, w_2 \in R$ be such that $w_1 \cdot w_2 = x$ in $S$. Then one of the following holds:
\begin{enumerate}
\item For some $0 \leq k \leq n$, we have \begin{align*}
w_1 \equiv \overline{x}_1 \overline{x}_2 \cdots \overline{x}_k, \quad \textnormal{and} \quad w_2 \equiv \overline{x}_{k+1} \overline{x}_{k+2} \cdots \overline{x}_n,
\end{align*}
where $\overline{x}_j \in R_{X(j)}$ and $\overline{x}_j = x_j$ in $S_{X(j)}$ for all $0 \leq j \leq n$.
\item For some $0 \leq k \leq n$, we have 
\begin{align*}
w_1 \equiv \overline{x}_1 \overline{x}_2 \cdots \overline{x}_{k-1} x_k' \quad \textnormal{and} \quad w_2 \equiv x_k'' \overline{x}_{k+1} \cdots \overline{x}_n,
\end{align*}
where $x_k', x_k'' \in R_{X(k)}$ with $x_k = x_k' x_k''$ in $S_{X(k)}$, and $\overline{x}_j \in R_{X(j)}$ with $\overline{x}_j = x_j$ in $S_{X(j)}$ for all $0 \leq j < k$ and $k < j \leq n$.
\end{enumerate}
\end{lemma}
\begin{proof}
This follows directly from Lemma~\ref{Lem:SGPFP-combing-equality} and the multiplication \eqref{Eq:sgp_fp_multiplication} in semigroup free products; case (1) corresponds to the first case of \eqref{Eq:sgp_fp_multiplication}, and case (2) corresponds to the second.
\end{proof}

We give a similar treatment regarding combings and \textit{monoid} free products. Let $M_1$ and $M_2$ be two monoids, generated by two finite disjoint sets $A_1$ resp. $A_2$. Let $M = M_1 \ast M_2$ denote their \textit{monoid} free product, and let $S$ denote their \textit{semigroup} free product.\footnote{To emphasise just how different $M$ and $S$ are, we note that $S$ is \textit{always} (!) an infinite semigroup, even if $M_1$ and $M_2$ are trivial monoids, whereas in this latter case $M$ would simply be trivial.} We let $A = A_1 \cup A_2$. Let $R_1, R_2$ be regular languages with $R_1 \subseteq A_1^\ast$ and $R_2 \subseteq A_2^\ast$, and with $R_1 \cap R_2 = \{ \varepsilon \}$. We say that a non-empty word $u \in \Alt(R_1, R_2)$, with alternating factorisation $u \equiv u_0 u_1 \cdots u_n$ and associated parametrisation $X$, is \textit{reduced} if $u_i \neq 1$ in $M_{X(i)}$ for all $0 \leq i \leq n$. The empty word is also declared to be reduced. Just as in the case of semigroup free products (Lemma~\ref{Lem:SGPFP-alt-is-combing}), it is easy to see that if $R_1, R_2$ are regular combings of $M_1$ resp. $M_2$, then $\Alt(R_1, R_2)$ is a regular combing of $M$. We write $R = \Alt(R_1, R_2)$. We have the following simple structural lemma, based on the identification of $M$ with the semigroup free product of $M_1$ by $M_2$ amalgamated over the trivial submonoid (see e.g. \cite[p. 266]{Howie1995}):

\begin{lemma}\label{Lem:Monoid-reduced-fp-NFT}
Let $u, v \in R \setminus \{ \varepsilon \}$ be reduced words. Then $u =_M v$ if and only if $u =_S v$.
\end{lemma}

Of course, this lemma would fail spectacularly if the \textit{reduced} condition is removed. Despite this connection between $M$ and $S$, there is one important distinction to make from the semigroup free product case: when multiplying the alternating word $u_0 u_1 \cdots u_k$ by the alternating word $v_0 v_1 \cdots v_n$ in a monoid free product, if we are in the case $u_k \sim v_0$ we may, of course, have $u_k v_0 =_{M_i} 1$ for $i=1$ or $2$. Unlike the case of semigroup free products, this now means that $u_k v_0 =_M 1$. Hence the multiplication table for $M$ with respect to the regular combing $R$ is mostly made up of the multiplication table for $S$, but with one additional case. We spell the above out in somewhat more technical language below:

\begin{lemma}\label{Lem:Mult_table_MON_lemma}
Let $x \equiv x_1 x_2 \cdots x_n \in R$ be reduced, with $x_i \in R_{X(i)}$ for some standard parametrisation $X$. Let $w_1, w_2 \in R$ be reduced with $w_1 \cdot w_2 = x$ in $M$. Then one of the following holds:
\begin{enumerate}
\item For some $0 \leq k \leq n$, we have \begin{align*}
w_1 \equiv \overline{x}_1 \overline{x}_2 \cdots \overline{x}_k, \quad \textnormal{and} \quad w_2 \equiv \overline{x}_{k+1} \overline{x}_{k+2} \cdots \overline{x}_n,
\end{align*}
where $\overline{x}_j \in R_{X(j)}$ and $\overline{x}_j = x_j$ in $M_{X(j)}$ for all $0 \leq j \leq n$.
\item For some $0 \leq k \leq n$, we have 
\begin{align*}
w_1 \equiv \overline{x}_1 \overline{x}_2 \cdots \overline{x}_{k-1} x_k' \quad \textnormal{and} \quad w_2 \equiv x_k'' \overline{x}_{k+1} \cdots \overline{x}_n,
\end{align*}
where $x_k', x_k'' \in R_{X(k)}$ with $1 \neq x_k = x_k' x_k''$ in $M_{X(k)}$, and $\overline{x}_j \in R_{X(j)}$ with $\overline{x}_j = x_j$ in $M_{X(j)}$ for all $0 \leq j < k$ and $k < j \leq n$.
\item For some $k \geq 0$ and $m \geq n$, we have 
\begin{align*}
w_1 \equiv \overline{x}_1 \overline{x}_2 \cdots \overline{x}_{k-1} x_k' \quad \textnormal{and} \quad w_2 \equiv x_k'' \overline{x}_{k+1} \cdots \overline{x}_m,
\end{align*}
where $x_k', x_k'' \in R_{X(k)}$ with $x_k' x_k'' = 1$ in $M_{X(k)}$. Furthermore, setting 
\begin{align*}
w_1' \equiv \overline{x}_1 \overline{x}_2 \cdots \overline{x}_{k-1} \quad \text{and} \quad w_2' \equiv \overline{x}_{k+1} \overline{x}_{k+2} \cdots \overline{x}_m,
\end{align*}
we have $w_1', w_2' \in R$ and $w_1' \cdot w_2' = x$ in $M$.
\end{enumerate}
\end{lemma}

Cases (1) and (2) are ``inherited'' from $S$ via combining Lemma~\ref{Lem:Mult_table_SGP_lemma} and Lemma~\ref{Lem:Monoid-reduced-fp-NFT} in the case that the concatenation $w_1 w_2$ is reduced, while case (3) corresponds to case (3) in Lemma~\ref{Lem:Monoid-reduced-fp-NFT}. This case (3) highlights the recursive nature of reduction in free products (cf. e.g. free reduction), and this recursion eventually terminates as $|w_j'|<|w_j|$ for $j=1, 2$. We give an example of Lemma~\ref{Lem:Mult_table_MON_lemma} below, in the case of the free product of two copies of the bicyclic monoid. 

\begin{example}
Let $M_i = \pres{Mon}{b_i, c_i}{b_ic_i = 1}$ for $i=1,2$ be two copies of the bicyclic monoid, and let $R_i = c_i^\ast b_i^\ast$. Let $x \equiv b_2^2$, and let $w_1 \equiv b_2 b_1$, $w_2 \equiv c_1 b_2$. Then 
\[
w_1 w_2 = b_2 b_1 c_1 b_2 = b_2^2 \equiv x,
\]
in $M$, so we can apply Lemma~\ref{Lem:Mult_table_MON_lemma}. Indeed, we find that we are in case (3), taking $k=2$ and $m=3$, $\overline{x}_1 \equiv b_2, x_2' \equiv b_1$ and $x_2'' \equiv c_1, \overline{x}_3 \equiv b_2$, for then 
\[
x_k' x_k'' \equiv x_2' x_2'' \equiv b_1 c_1 = 1
\]
in $M_1$, and we have $w_1' \equiv \overline{x}_1 \equiv b_2$ and $w_2' \equiv \overline{x}_3 \equiv b_2$, and this satisfies $w_1' \cdot w_2' \equiv b_2^2 =_M x$. We may reapply Lemma~\ref{Lem:Mult_table_MON_lemma}, and find ourselves in case (1), taking $\overline{x}_1 \equiv b_2$ and $\overline{x}_n \equiv b_2$. 
\end{example}

These were all the statements we shall require about free products for the sequel.

\subsection{$\ETOL$ and substitutions}\label{Subsec:ETOL-subst}

Word-hyperbolicity is intricately connected with $\CF$-multiplication tables. However, our results will be true more generally, substituting e.g. $\ETOL$ or $\IND$ for $\CF$, and we will elaborate on this topic in \S\ref{Sec:Finalremarks}. Specifically, the proofs of the main results about preservation properties in free products of word-hyperbolic algebraic structures (semigroups, monoids, or groups) in \S\ref{Sec:SGPFP} and \S\ref{Sec:MFP} will all be applicable to free products of algebraic structures with $\cc$-multiplication tables, where $\cc$ is some full $\AFL$ satisfying the monadic ancestor property. This will include the cases when $\cc$ is one of $\CF, \IND$, or $\ETOL$. We will give a brief overview of the strong historical connections between $\ETOL$ and the monadic ancestor property. This is a complex history; we cannot do it full justice here, and it will be expanded on in a future survey article. 

We give the definition of a substitution. Let $A$ be an alphabet. For each $a \in A$, let $\sigma(a)$ be a language (over any finite alphabet); let $\sigma(\varepsilon) = \{ \varepsilon\}$; for every $x, y \in A^\ast$ let $\sigma(xy) = \sigma(x)\sigma(y)$; and for every $L \subseteq A^\ast$, let $\sigma(L) = \bigcup_{w\in L} \sigma(w)$. We then say that $\sigma$ is a \textit{substitution}. For a class $\cc$ of languages, if for every $a \in A$ we have $\sigma(a) \in \cc$, then we say that $\sigma$ is a $\cc$-\textit{substitution}. Let $A$ be an alphabet, and $\sigma$ a substitution on $A$. For every $a \in A$, let $A_a$ denote the smallest finite alphabet such that $\sigma(a) \subseteq A_a^\ast$. Extend $\sigma$ to $A \cup (\bigcup_{a \in A}A_a)$ by defining $\sigma(b) = \{ b \}$ whenever $b \in (\bigcup_{a \in A} A_a) \setminus A$. For $L \subseteq A^\ast$, let $\sigma^1(L) = \sigma(L)$, and let $\sigma^{n+1}(L) = \sigma(\sigma^n(L))$ for $n \geq 1$. Let $\sigma^\infty(L) = \bigcup_{n> 0} \sigma^n(L)$. Then we say that $\sigma^\infty$ is an \textit{iterated substitution}. If for every $b \in A \cup (\bigcup_a A_a)$ we have $b \in \sigma(b)$, then we say that $\sigma^\infty$ is a \textit{nested} iterated substitution. If $\sigma^\infty$ is nested, then it is convenient for inductive purposes to set $\sigma^0(L) := L$. Note that the nested property ensures $L \subseteq \sigma(L)$, so $\bigcup_{n \geq 0} \sigma^n(L) = \bigcup_{n>0} \sigma^n(L)$. We say that $\cc$ is closed under nested iterated substitution if for every $\cc$-substitution $\sigma$ and every $L \in \cc$, we have: if $\sigma^\infty$ is a nested iterated substitution, then $\sigma^\infty(L) \in \cc$. A similar definition yields closure under iterated substitutions. For the benefit of the reader, we mention two facts, sourced below: the class $\CF$ is closed under nested iterated substitution (but not iterated substitution), and the class $\ETOL$ is closed under iterated substitution (and hence also nested iterated substitution). 

Substitutions are closely related to $\AFL$s. Indeed, the 1967 article by Greibach \& Ginsburg which first defined $\AFL$s \cite{Ginsburg1967} (later expanded in \cite{Ginsburg1969a}) included a proof about a form of substitution-closure for $\AFL$s (under $\varepsilon$-free regular substitutions) and for full $\AFL$s (under regular substitutions). The closure of $\CF$ under nested iterated substitution was proved by Kr\'al \cite{Kral1970} in 1970.  Following some further results (e.g. \cite{Ginsburg1970b}), an abstract basis for substitution was developed by Ginsburg \& Spanier \cite{Ginsburg1970}; one particular important notion developed there was treating the (nested) substitution-closure of a full $\AFL$ as a form of ``algebraic closure''. In particular, it is proved that the substitution-closure of a full $\AFL$ is a full $\AFL$ \cite[Theorem~2.1]{Ginsburg1970}. Lewis \cite{Lewis1970} used substitution to \textit{define} full $\AFL$s, and rediscovers the aforementioned result by Ginsburg \& Spanier, see \cite[Theorem~1.13]{Lewis1970}. See also \cite{Asveld1977, Beauquier1981}.

Substitutions can be useful in studying full $\AFL$s for a number of reasons; for example, one can recover results of Ginsburg \& Greibach \cite{Ginsburg1970b} about principal $\AFL$s, see \cite[Corollary~1.21]{Lewis1970}. One can also use substitution-based ideas to produce (see \cite[Corollary~4.13]{Christensen1974}) an infinite strict hierarchy 
\[
\CF \subsetneq \cc_1 \subsetneq \cc_2 \subsetneq \cdots \subsetneq \cc_i \subsetneq \cdots \subsetneq \ETOL
\]
of full $\AFL$-s $\cc_i$ between the classes $\CF$ and $\ETOL$, cf. also \cite{Greibach1970b, Latteux1979} for related such hierarchies; for similar hierarchies between $\ETOL \subset \IND$, see \cite{Ehrenfeucht1975, Engelfriet1985}; and for infinite hierarchies between $\IND \subset \CS$, see \cite{Asveld1975, Engelfriet1981}. 

Because of the importance and utility of iterated substitution, Greibach \cite{Greibach1969b} (later expanded in \cite{Greibach1970}) defined \textit{super}-$\AFL$s as a full $\AFL$ closed under nested iterated substitution (by \cite[Proposition~2.2]{NybergBrodda2021f}, this is equivalent to the definition of super-$\AFL$ as defined in \S\ref{Subsec:Rewritingsystems}). Not long after, the notion of a \textit{hyper}-$\AFL$ was introduced, being any full $\AFL$ closed under iterated substitution \cite{Salomaa1974, Asveld1977b}.\footnote{P. R. J. Asveld \cite[p. 1]{Asveld2001} on this point says the following: ``Similar as in ordinary algebra -- where one went from groups to semigroups, rings, and fields -- full $\AFL$s gave rise to weaker structures (full trios, full semi$\AFL$s) and more powerful ones: full substitution-closed $\AFL$s, full super-$\AFL$s, and full hyper-$\AFL$s.'' We cannot agree with this assessment of the historical development of ``ordinary'' algebra. Finite fields and groups were intricately connected already in the early works of both Lagrange and Galois (cf. \cite{Neumann2011}), whereas rings and semigroups would not appear as objects of study until half a century resp. a century later. Similarly, Klein initially posed an axiomatisation of \textit{group} as what we today call a \textit{monoid}, but as Lie ``in his study of infinite groups saw it as necessary to expressly require [the existence of inverses]'', it was this axiomatisation that was chosen (``...sah sich Lie gen\"otigt, ausdr\"ucklich zu verlangen...'', \cite[p. 335]{Klein1979}). We strongly recommend the interested reader to consult Wu{\ss}ing \cite{Wussing1969}. The above paragraph shows the difficulty in simplifying the development of ordinary algebra in a linear manner; and one may feel similarly about the linear narrative regarding $\AFL$s.} Many fundamental results about hyper-$\AFL$s and substitution were developed by Christensen \cite{Christensen1974}, who also, along with Asveld \cite{Asveld1977b}, fleshed out the connections between $\ETOL$ and hyper-$\AFL$s noted by e.g. Salomaa \cite{Salomaa1973, Salomaa1974} and \v{C}ul\'ik \cite{Culik1975}; cf. also \cite{Engelfriet1976, Drewes2004}. In particular, at this point we arrive at the following rather pleasant result: 

\begin{theorem}\label{Thm:CF_ETOL_smallest_AFLs}
\begin{enumerate}
\item{\cite[Theorem~2.2]{Greibach1970}} The class $\CF$ is the least super-$\AFL$.
\item{\cite[Corollary~4.10]{Christensen1974}} The class $\ETOL$ is the least hyper-$\AFL$.
\end{enumerate}
\end{theorem}

Furthermore, one can also show that $\IND$ is a super-$\AFL$ \cite{Downey1974}.

As mentioned, the connections between substitutions and $\ETOL$ remain active research topics (if somewhat implicitly), but are far too numerous to recount here. While they will be given a proper treatment in the future, we mention a few. For example, one can give a complexity analysis of iterated substitutions, with applications to both $\ETOL$ and $\EDTOL$ languages \cite{Asveld1987}, and there are connections with fuzzy logic \cite{Asveld2003}. One can also extend the notion of substitution to ``deterministic substitution'' (which will not be defined here), leading to a statement analogous to the fact that $\EDTOL$ is the least dhyper-$\AFL$ \cite[Corollary~4.5]{Asveld1977}; see also \cite{Latteux1977} for more on $\EDTOL$ and substitutions. At this point, it bears mentioning that there is a great deal of involved and often obfuscating notation and abbreviations; as an example, we have that ``if $K$ is a pseudoid, then $\eta(K)$ is the smallest full dhyper-$\operatorname{QAFL}$ containing $K$'' \cite[Theorem~4.5]{Asveld1977}. In addition, there are a great number of abbreviations for classes of languages associated with Lindenmayer systems (yielding the $\mathbf{L}$); aside from $\ETOL$ and $\EDTOL$, we have e.g. 
\[
\mathbf{L}, \mathbf{0L}, \mathbf{P0L}, \mathbf{T0L}, \mathbf{E0L}, \mathbf{X0L}, \mathbf{EP0L}, \mathbf{FE0L}(k), \mathbf{EPT0L}, \mathbf{FEPT0L}(k), \dots 
\]
see e.g. \cite{Meduna2003} for a large number of these.\footnote{Given the number of abbreviations, one may reasonably inquire about the language-theoretic properties of the language of all abbreviations of classes of languages.} We ensure the reader not familiar with this multitude of notation that most, if not all, such classes are generally defined (or definable) by relatively straightforward means; cf. e.g. the definition of $\ETOL$ as given by Theorem~\ref{Thm:CF_ETOL_smallest_AFLs}(2). Furthermore, the reader may notice, in the subsequent sections, the importance of substitution in dealing with free products -- this link between the algebraic and the formal language theoretic runs deep, and there seems to be ample opportunity to develop it further.

\section{Free products of word-hyperbolic semigroups}\label{Sec:SGPFP}

In this section, we will prove the main result regarding semigroup free products and word-hyperbolicity (Theorem~\ref{Thm:Sgp_extendable_result}). 

Let $S_1, S_2$ be two semigroups, finitely generated by disjoint sets $A_1$ resp. $A_2$ and with regular combings $R_1$ resp. $R_2$. Let $S = S_1 \ast S_2$ denote the semigroup free product of $S_1$ and $S_2$. We begin by recalling (Lemma~\ref{Lem:SGPFP-alt-is-combing}) that the language $\Alt^+(R_1, R_2)$ of alternating words is a combing for $S$. Let $R = \Alt^+(R_1, R_2)$. This will, in the sequel, be our chosen combing for proving that the table $\cT_S(R)$ is context-free when the factors $S_1, S_2$ are word-hyperbolic.

\begin{theoremB}
Let $S_1, S_2$ be $1$-extendable word-hyperbolic semigroups. Then the free product $S_1 \ast S_2$ is word-hyperbolic. 
\end{theoremB}
\begin{proof}
Suppose, for $i=1, 2$, that $S_i$ is generated by the finite set $A_i$, and that $S_i$ is word-hyperbolic with respect to the regular combing $R_i \subseteq A_i^+$, with the multiplication table $\mathcal{T}(R_i)$ context-free. We assume without loss of generality that $A_1 \cap A_2 = \varnothing$, and hence that $R_1 \cap R_2 = \varnothing$. As $S_i$ is $1$-extendable, the semigroup $S_i^\bi$ is word-hyperbolic with respect to the regular combing $\overline{R_i}  = R_i \cup \{ \varepsilon \}$, where now $\varepsilon$ is the unique word mapping to the identity element $\bi_i$ of $S_i^\bi$. Let $A = A_1 \cup A_2$.

For $i=1, 2$, define the monadic rewriting system  $\sR_i$ by 
\begin{equation}
\sR_i = \{ (w, \#_2) \mid w \in \mathcal{T}(\overline{R_i}) \}.
\end{equation}
Then by assumption $\sR_i$ is a context-free monadic rewriting system. Note that for every $x, x' \in R_i$ with $x =_{S_i} x'$ we have that $x' \#_1 \#_2 x^\trev \in \mathcal{T}(\overline{R_i})$, as $\pi_i(x') \cdot \bi_i = \pi_i(x)$, and thus also $(x' \#_1 \#_2 x^\trev, \#_2) \in \sR_i$. Let $\sR$ be the rewriting system $\sR_1 \cup \sR_2$. This is also a context-free monadic rewriting system. Recall that $A_\# = A \cup \{ \#_1, \#_2 \}$. 

\begin{lemma}\label{Lem:Anc-of-R_i}
Let $w \in A^\ast_\#$. Then $w \in \anc{\sR}(\#_1 \#_2)$ if and only if it is of the form 
\begin{equation}\label{Eq:general-form-of-R-anc}
w \equiv \#_1 \left( \prod_{i=1}^n x_i \#_1 y_i \right) \#_2 \left(\prod_{i=1}^n z_i \right)^\trev,
\end{equation}
for some $n \geq 0$, and where for every $1 \leq i \leq n$, we have $x_i, y_i, z_i \in \overline{R}_{X(i)}$ with $x_i y_i = z_i$ in $S^\bi_{X(i)}$, where $X$ is some parametrisation.
\end{lemma}
\begin{proof}
For ease of notation, we write $\xra{}$ for $\xra{\sR}$, and analogously for $\xr{}, \xr{}^k$, etc. 

$(\impliedby)$ Suppose $w$ is of the form \eqref{Eq:general-form-of-R-anc}. We prove the claim by induction on $n$. The case $n=0$ is immediate. Suppose $n > 0$. Then $w$ contains exactly one occurrence of $\#_2$; to the left of this occurrence is an occurrence of the word $x_n \#_1 y_n$, and to the right is an occurrence of the word $z_n^\trev$. As $x_n, y_n, z_n \in \overline{R}_{X(n)}$ and $x_n y_n = z_n$ in $S_{X(n)}^\bi$, we have $(x_n \#_1 y_n \#_2 z_n^\trev, \#_2) \in \sR_{X(n)} \subseteq \sR$. Hence 
\[
w \xr{} \#_1 \left( \prod_{i=1}^{n-1} x_i \#_1 y_i \right) \#_2 \left( \prod_{i=1}^{n-1} z_i \right)^\trev,
\]
and the right-hand side now lies in $\anc{\sR}(\#_1 \#_2)$ by the inductive hypothesis. 

$(\implies)$ Suppose $w \xra{} \#_1 \#_2$, say $w \xr{}^k \#_1 \#_2$ for some $k \geq 0$. The proof is by induction on $k$. The base case $k=0$ is trivial, for then $w \equiv \#_1 \#_2$. Suppose $k>0$. Then there is some $w' \in A^\ast_\#$ such that $w \xr{} w' \xr{}^{k-1} \#_1 \#_2$, and such that the rewriting is via some rule $r \equiv (x\#_1 y \#_2 z^\trev, \#_2) \in \sR$. Then as $r \in \sR$, we have $x, y, z \in \overline{R}_1 \cup \overline{R}_2$ and $x \cdot y = z$ in $S^\bi_j$ for $j=1$ or $j=2$. Now, by the inductive hypothesis,
\[
w' \equiv \#_1 \left( \prod_{i=1}^m x_i \#_1 y_i \right) \#_2 \left( \prod_{i=1}^m z_i\right)^\trev,
\]
with some parametrisation $X'$ such that for every $1 \leq i \leq m$, we have $x_i' y_i' =_{} z'_i$ in $S^\bi_{X'(i)}$. As the right-hand side of $r$ contains only one occurrence of $\#_2$, and as $w'$ contains only one occurrence of $\#_2$, it follows that 
\begin{equation}\label{Eq:w=complicated_particular_lemma1}
w \equiv \#_1 \left( \prod_{i=1}^m x_i \#_1 y_i \right) (x \#_1 y \#_2 z^\trev) \left( \prod_{i=1}^m z_i\right)^\trev,
\end{equation}
and hence, taking $n = m+1$ and defining the parametrisation $X(i) = X'(i)$ for $i \neq n$, and $X(n) = j$, the expression \eqref{Eq:w=complicated_particular_lemma1} is an expression of the form \eqref{Eq:general-form-of-R-anc} for $w$. 
\end{proof}

We will now show that a rational transduction of the language of all words of the form \eqref{Eq:general-form-of-R-anc} equals $\cT_S(R)$. This will yield the result. Let $\tau_0 \subseteq A_\#^\ast \times A_\#^\ast$ be the rational transduction defined by 
\[
\tau_0 = \bigg\{ \big\{ (\#_1, \#_1), (\#_1, \varepsilon)\big\} \cup \big\{ (a, a) \mid a \in A \cup \{ \#_2 \}\big\}\bigg\}^\ast
\]
For any word $w \in A_\#^\ast$, the language $\tau_0(w)$ consists of all words obtainable by erasing some (possibly zero) amount of $\#_1$-symbols in $w$, while fixing all other symbols. Define the language 
\begin{equation}\label{Eq:Useful_expr_for_table}
\cl_0 = \tau_0 \left( \anc{\sR}(\#_1 \#_2) \right) \cap R \#_1 R \#_2 R^\trev.
\end{equation} 

\begin{lemma}\label{Lem:cl0_is_CF}
The language $\cl_0$ is a context-free language. 
\end{lemma}
\begin{proof}
This is an immediate consequence of the expression \eqref{Eq:Useful_expr_for_table}, in combination with the facts that (i) $\sR$ is a context-free monadic rewriting system; (ii) every singleton language is in $\CF$; (iii) the class $\CF$ has the monadic ancestor property; and (iv) the class $\CF$ closed under rational transduction (and hence also, in particular, intersection with regular languages).  
\end{proof}

We now show that $\cl_0 = \cT_S(R)$.

\begin{lemma}\label{Lem:w_anc_is_in_table}
$\cl_0 \subseteq \cT_S(R)$. 
\end{lemma}
\begin{proof}
Suppose $w \in \cl_0$. Then (1) $w$ is an element of $\tau_0(w')$, where $w'$ is of the form \eqref{Eq:general-form-of-R-anc} (by Lemma~\ref{Lem:Anc-of-R_i}); and (2) $w \in R \#_1 R \#_2 R^\trev$. As $\tau_0(w')$ consists of all words obtainable from $w'$ by erasing some number of $\#_1$-symbols, and the words in $R \#_1 R \#_2 R^\trev$ contain exactly one $\#_1$, it follows from the expression \eqref{Eq:general-form-of-R-anc} for $w'$ that 
\begin{equation}\label{Eq:L=>T,expression_for_w}
w \equiv \left( \prod_{i=1}^k x_i y_i \right) \#_1 \left( \prod_{i=k+1}^n x_i y_i \right) \#_2 \left( \prod_{i=1}^n z_i \right)^\trev
\end{equation}
where for every $1 \leq i \leq n$, we have $x_i, y_i, z_i \in \overline{R}_{X(i)}$ and $x_i y_i = z_i$ in $S^\bi_{X(i)}$, with $X$ some parametrisation. 

Now $x_i \sim y_i$ for all $1 \leq i \leq n$. Furthermore, $y_i \not\sim x_{i+1}$ for all $1 \leq i < k$ and $k< i \leq n$, as $\prod_{i=1}^k (x_iy_i)$ resp. $\prod_{i=k+1}^n (x_iy_i)$ are alternating words. It follows that we must have $x_iy_i \in R_{X(i)}$ for every $1 \leq i \leq n$ and that $z_i \not\sim z_{i+1}$ for every $1 \leq i < n$ except possibly $i=k$. We thus have two cases: (1) $z_k \not\sim z_{k+1}$, or else (2) $z_k \sim z_{k+1}$. In either case, let $\overline{z}_i \equiv x_i y_i$ for $1 \leq i \leq n$. Then $\overline{z}_i \in R_{X(i)}$, and $w$ is of the form 
\[
w \equiv \left( \prod_{i=1}^k \overline{z}_i \right) \#_1 \left( \prod_{i=k+1}^n \overline{z}_i \right) \#_2 \left( \prod_{i=1}^n z_i \right)^\trev.
\]
Suppose we are in case (1). As $\overline{z}_i \equiv x_iy_i$, and $\overline{z}_i = z_i$ in $S^\bi_{X(i)}$, we thus have that $w$ is the element of the multiplication table $\cT_S(R)$ corresponding to the product 
\[
(\overline{z}_1 \overline{z}_2 \cdots \overline{z}_k) \cdot (\overline{z}_{k+1} \overline{z}_{k+2} \cdots \overline{z}_n) =_S z_1 z_2 \cdots z_n,
\]
which clearly holds in $S$.

In case (2), as $z_k =_{S^\bi_{X(k)}} x_k y_k \equiv \overline{z}_k \sim \overline{z}_{k+1} =_{S^\bi_{X(k+1)}} z_{k+1}$, it follows that $z_k \sim z_{k+1}$, and hence, as $R$ consists of alternating words, that $z_k z_{k+1} \in R_{X(k)}$. Let $z \equiv z_k z_{k+1}$, and let $\overline{z} \equiv \overline{z}_k \overline{z}_{k+1}$. Then $\overline{z} =_{S^\bi_{X(k)}} z$. Thus $w$ is the element of $\cT_S(R)$ corresponding to the product 
\[
(\overline{z}_1 \overline{z}_2 \cdots \overline{z}_k) \cdot (\overline{z}_{k+1} \overline{z}_{k+2} \cdots \overline{z}_n) =_S z_1 z_2 \cdots z_{k-1} z z_{k+2} \cdots z_n,
\]
which also clearly holds in $S$. Thus, in either case, we have that $w \in \cT_S(R)$.
\end{proof}

We hence have $\cl_0 \subseteq \cT_S(R)$. We now prove the converse of Lemma~\ref{Lem:w_anc_is_in_table}.

\begin{lemma}\label{Lem:w_in_TS_w_inL0}
$\cT_S(R) \subseteq \cl_0$. 
\end{lemma}
\begin{proof}
Suppose that $w \equiv w_1 \#_1 w_2 \#_2 x^\trev \in \cT_S(R)$, i.e. that $w_1, w_2, x \in R$ are such that $w_1 \cdot w_2 =_S x$. As $w_1, w_2, x \in R = \Alt^+(R_1, R_2)$, we have that 
\[
x \equiv x_1 x_2 \cdots x_n
\]
where $x_i \in R_{X(i)}$ for some standard parametrisation $X$. By Lemma~\ref{Lem:Mult_table_SGP_lemma}, we either fall in case (1) or (2) of the same lemma. 

In case (1), we have, using the notation of that lemma, that 
\begin{align*}
w &\equiv w_1 \#_1 w_2 \#_2 x^\trev \\
&\equiv \overline{x}_1 \overline{x}_2 \cdots \overline{x}_k \#_1 \overline{x}_{k+1} \overline{x}_{k+2} \cdots \overline{x}_{n} \#_2 x_n^\trev x_{n-1}^\trev \cdots x_1^\trev \\
&\in \tau_0 \left[ \left( \prod_{i=1}^k \#_1 \overline{x}_i \right)\left( \prod_{i=k+1}^n \#_1 \overline{x}_i \right) \#_1  \#_2 \left( \prod_{i=i}^n x_i \right)^\trev \right] \\
&= \tau_0 \left[ \left( \prod_{i=1}^n \#_1 \overline{x}_i \right) \#_1 \#_2 \left( \prod_{i=1}^n x_i \right)^\trev \right].
\end{align*}
Let $W \equiv \left( \prod_{i=1}^n \#_1 \overline{x}_i \right) \#_1 \#_2 \left( \prod_{i=i}^n x_i \right)^\trev$. As $w_1, w_2, x \in R$, and hence $w \in R \#_1 R \#_2 R^\trev$, it suffices by the expression \eqref{Eq:Useful_expr_for_table} to show that $W \in \anc{\sR}(\#_1 \#_2)$. As $\overline{x}_i = x_i$ in $S_{X(i)}$, we have $(\overline{x}_i \#_1 \#_2 x_i^\trev, \#_2) \in \sR$ for every $1 \leq i \leq n$. Hence 
\[
W \xr{} \left( \prod_{i=1}^{n-1} \#_1 \overline{x}_i \right) \#_1 \#_2 \left( \prod_{i=1}^{n-1} x_i \right)^\trev \xr{} \cdots \xr{} (\#_1 \overline{x}_1) \#_1 \#_2 (x_1^\trev) \xr{} \#_1 \#_2,
\]
which is what was to be shown. 

In case (2), the proof is almost the same as in case (1), but the reductions are no longer exclusively by rules of the form $(\overline{x}_i \#_1 \#_2 x_i^\trev, \#_2)$. In the same way as in case (1), however, we find that $w \in \tau_0(W)$, where 
\[
W \equiv \left( \prod_{i=1}^{k-1} \#_1 \overline{x}_i \right) \#_1 (\overline{x}_k' \#_1 \overline{x}_k'') \left( \prod_{k+1}^n \#_1 \overline{x}_i\right) \#_2 \left( \prod_{i=1}^n x_i \right)^\trev.
\] 
By applying the rules $(\overline{x}_i \#_1 \#_2 x_i^\trev \to \#_2)$ to $W$, for $i = n, n-1, \dots, k+1$ (all such rules are in $\sR$ as $\overline{x}_i \cdot \varepsilon = x_i$ in $S_{X(i)}^\bi$), we find that
\begin{align*}
W \xra{} &\left( \prod_{i=1}^{k-1} \#_1 \overline{x}_i \right) \#_1 (\overline{x}_k' \#_1 \overline{x}_k'') \#_2 \left( \prod_{i=1}^k x_i \right)^\trev \\
\equiv &\left( \prod_{i=1}^{k-1} \#_1 \overline{x}_i \right) \#_1 (\overline{x}_k' \#_1 \overline{x}_k'') \#_2 x_k^\trev \left( \prod_{i=1}^{k-1} x_i \right)^\trev \\
\xr{} &\left( \prod_{i=1}^{k-1} \#_1 \overline{x}_i \right) \#_1 \#_2 \left( \prod_{i=1}^{k-1} x_i \right)^\trev,
\end{align*}
where in the final step we used the rule $(\overline{x}_k' \#_1 \overline{x}_k'' \#_2 x_k^\trev, \#_2)$, which is in $\sR$ as $\overline{x}_k' \cdot \overline{x}_k'' = x_k^\trev$ in $S_{X(k)}$ (and hence also in $S_{X(k)}^\bi$).  The proof now proceeds just as in case (1), and we find that $W \in \anc{\sR}(\#_1 \#_2)$, and as $w \in \tau_0(W)$ and $w \in R \#_1 R \#_2 R^\trev$, we have $w \in \cl_0$. 
\end{proof}

Thus, we have $\cT_S(R) = \cl_0$. As $R$ is a regular combing of $S = S_1 \ast S_2$ by Lemma~\ref{Lem:SGPFP-alt-is-combing}, and as $\cl_0$ is context-free by Lemma~\ref{Lem:cl0_is_CF}, we conclude that $(R, \cT_S(R))$ is a word-hyperbolic structure for $S = S_1 \ast S_2$. This completes the proof of Theorem~\ref{Thm:Sgp_extendable_result}.
\end{proof}

By Lemma~\ref{Lem:Types-of-extendable}, we find the following explicit corollaries of Theorem~\ref{Thm:Sgp_extendable_result}:

\begin{corollary}\label{Cor:Monoids_sgp}
The semigroup free product of two word-hyperbolic monoids is word-hyperbolic. 
\end{corollary}

\begin{corollary}\label{Cor:Regular-sgp-fp}
The semigroup free product of two (von Neumann) regular word-hyperbolic semigroups is word-hyperbolic.
\end{corollary}

\begin{corollary}\label{Cor:Uniqueness-sgp-fp}
The semigroup free product of two word-hyperbolic semigroups with uniqueness is word-hyperbolic with uniqueness.
\end{corollary}

The final ``with uniqueness'' in the statement of Corollary~\ref{Cor:Uniqueness-sgp-fp} follows from the fact that the elements of $\Alt^+(R_1, R_2)$ represent pairwise distinct elements of $S$. We will now turn towards considering monoid free products. To do this, we first need to introduce a useful purely language-theoretic operation.

\section{Polypartisan Ancestors}\label{Sec:Polypartisan}

\noindent In this section, we will generalise (in a fairly uncomplicated manner) the \textit{bipartisan ancestors} introduced in \cite{NybergBrodda2021f} to \textit{poly}partisan ancestors, and prove that this construction preserves certain language-theoretic properties of the languages it is applied to. We will use this construction to obtain the multiplication table for a monoid free product from the table for a semigroup free product.

Let $A$ be a finite alphabet, and let $k \geq 1$. Let $\#_1, \#_2, \dots, \#_{k}$ be $k$ new symbols, and let $A_\# = A \cup \bigcup_{i=1}^{k} \{ \#_i \}$. We will let $\sh_k(A)$ denote the shuffle product
\[
\sh_k(A) =  A^\ast \sh\{ \#_1 \#_2 \cdots \#_k \}  = \{ u_0 \#_1 u_1 \#_2 \cdots \#_k u_k \mid u_i \in A^\ast\}.
\]
We call $\sh_k(A)$ the \textit{full $k$-shuffled language} (associated to $A$). Any subset of $\sh_k(A)$ is called a $k$-shuffled language (with respect to $A$). Thus, the ``word problem'' in the sense of Duncan \& Gilman \cite{Duncan2004} for a monoid generated by $A$ is a $1$-shuffled language, i.e. a subset of $\sh_1(A)$, and its multiplication table is a $2$-shuffled language, i.e. a subset of $\sh_2(A)$. Furthermore, the solution set for a set of equations in $k$ unknowns over a group is a $k$-shuffled language \cite{Ciobanu2016}. 

For elements $w \in \sh_k(A)$, we will introduce the notation 
\[
w \equiv [u_0, u_1, \dots, u_k] \quad \iff \quad w \equiv u_0 \#_1 u_1 \#_2 \cdots \#_k u_k.
\]
To abbreviate even further, we will write $[u_{(k)}]$ for $[u_0, u_1, \dots, u_k]$. Thus, the word problem for a monoid $M$ consists of words $[u_{(1)}]$ with $u_0 =_M u_1^\trev$, and a multiplication table for $M$ consists of words of the form $[v_{(2)}]$ with $v_0 \cdot v_1 =_M v_2^\trev$.

Let $k \geq 1$, and let $\sR_0, \sR_1, \dots, \sR_k \subseteq A^\ast \times A^\ast$ be a collection of $k+1$ rewriting systems. Let $L \subseteq \sh_k(A)$ be any language. We will define a new language $\sR_{(k)}(L) \subseteq \sh_k(A)$ as:
\begin{equation}\label{Eq:Def_of_k_ancestors}
\sR_{(k)}(L) = \{ [w_{(k)}] \colon \exists [u_{(k)}] \in L \textnormal{ such that } w_i \xra{\sR_i} u_i \textnormal{ for all $0 \leq i \leq k$} \}.
\end{equation}
We call $\sR_{(k)}(L)$ the \textit{$(k+1)$-partisan ancestor} of $L$ (with respect to $\sR_0, \sR_1, \dots, \sR_k$). A useful intuition for polypartisan ancestors is as follows: one may imagine the $k$-shuffled language 

Polypartisan ancestors generalise in an easy way the bipartisan ancestors introduced by the author in \cite{NybergBrodda2021f}. It is clear that $\sR_{(k)}(L)$ is a $k$-shuffled language. The use for polypartisan ancestors in this present article will be in preserving language-theoretic properties, in the following sense: 

\begin{proposition}\label{Prop:Polypartisan_preserves} 
Let $\cc$ be a super-$\AFL$. Let $L \in \cc$, and let $\sR_i \subseteq A^\ast \times A^\ast$ be $\cc$-monadic rewriting systems for $0 \leq i \leq k$. Then $\sR_{(k)}(L) \in \cc$.
\end{proposition}

The technique we will use to prove Proposition~\ref{Prop:Polypartisan_preserves} is a generalisation of a similar technique used to prove \cite[Proposition~2.5]{NybergBrodda2021f}, but follows its ideas rather closely. We will first prove a weaker form of Proposition~\ref{Prop:Polypartisan_preserves} (namely Lemma~\ref{Lem:Weak_polypartisan}). We will then use a rational transduction to move from the general case to this weaker form.

Let $A_0, A_1, \dots, A_k$ be $k+1$ alphabets, with $A \cap A_i = \varnothing$ for all $i$, and with $A_i \cap A_j = \varnothing$ for $i \neq j$. We let 

\[
\widehat{\sh}_k(A_0, \dots, A_k) = \sh_k(\bigcup_{i=0}^k A_i) \cap A_0^\ast \#_1 A_1^\ast \#_2 \cdots \#_k A_k^\ast.
\]
For this ``separated shuffle'', preservation properties are simple to prove.

\begin{lemma}\label{Lem:Weak_polypartisan}
Let $\sR_i \subseteq A_i^\ast \times A_i^\ast$ be $\cc$-monadic rewriting systems for $0 \leq i \leq k$. Let $L \in \cc$ be such that $L \subseteq \widehat{\sh}_k(A_0, \dots, A_k)$. Then $\sR_{(k)}(L) \in \cc$.
\end{lemma}
\begin{proof}
This closely follows the proof of \cite[Lemma~2.4]{NybergBrodda2021f}, which is the case for $k=1$, so we only sketch the main idea. As $\sR_{(k)}(L) \in \sh_k(A)$, it suffices to show that $\sR_{(k)}(L) \in \cc$. It is not difficult to see that as the alphabets $A_i$ are disjoint, $\sR_i \subseteq A_i^\ast \times A_i^\ast$, and every word in $L$ is of the form $u_0 \#_1 u_1 \#_2 \cdots \#_k u_k$, where $u_i \in A_i^\ast$, we have that $\sR_{(k)}(L) =\anc{\cup_{i=0}^k \sR_i}(L)$. As each $\sR_i$ is $\cc$-monadic -- and $\cc$ is closed under union being a super-$\AFL$ -- so too is $\sR := \cup_{i=0}^k \sR_i$. As $\cc$ is a super-$\AFL$, it has the monadic ancestor property, whence we find that $\anc{\sR}(L)$ is in $\cc$.
\end{proof}

We will from this point on assume that $|A_i| = |A|$ for all $0 \leq i \leq k$, and fix bijections $\varphi_i \colon A \to A_i$. We extend these to isomorphisms $\varphi_i \colon A^\ast \to A_i^\ast$ of free monoids. We will let $A_{I} = \bigcup_{i=0}^k A_i$, and let $A_{I,\#} = A_I \cup \bigcup_{i=0}^k\{ \#_i \}$. Further, we will write $\sR^{\varphi}_{i} = \varphi_i( \sR_i)$, where the action of $\varphi_i$ is entry-wise on the rules of $\sR_i$. If $\sR_i$ is a $\cc$-monadic rewriting system, then so too clearly is $\sR_i^\varphi$. 

We define a rational transduction $\mu_k \subseteq A_\#^\ast \times A_{I,\#}^\ast$ as
\begin{equation}\label{Eq:mu-trans-polypartisan}
\mu_k = \left( \bigcup_{a \in A} (a, \varphi_k(a) )\right)^\ast  \prod_{i=1}^k \left( \bigcup_{a \in A} (a, \varphi_i(a) )\right)^\ast (\#_i,\#_i).
\end{equation}
Then $\mu_k$ is indeed rational, as it is of the form $X_0^\ast x_1 X_1^\ast \cdots x_k X_k^\ast$, where the subset $X_i \subseteq A_\#^\ast \times A_{I,\#}^\ast$ is finite for $0 \leq i \leq k$, and $x_j \in A_\#^\ast \times A_{I,\#}^\ast$ is a single element for $1 \leq i \leq k$. Hence $\mu_k$ is a rational subset of $A_\#^\ast \times A_{I,\#}^\ast$. If $\mu_k$ is applied to (the singleton language containing) exactly one word $w \in \sh_k(A)$, it clearly produces (the singleton language containing) exactly one word from $\widehat{\sh}_k(A_0, \dots, A_k)$, and $\mu_k$ is injective on $\sh_k(A)$. That is, if 
\begin{equation*}
w \equiv u_0 \#_1 u_1 \#_2 \cdots \#_k u_k \in \sh_k(A),
\end{equation*}
where $u_i \in A^\ast$, then 
\begin{equation}\label{Eq:MuK=varphi}
\mu_k(w) = \{ \varphi_0(u_0) \#_1 \varphi_1(u_1) \#_2 \cdots \#_k \varphi_k(u_k) \},
\end{equation}
and if $w_1, w_2 \in \sh_k(A)$, then $\mu_k(w_1) = \mu_k(w_2)$ if and only if $w_1 \equiv w_2$, as each $\varphi_i$ is an isomorphism of free monoids. Slightly abusively, we will write the equality \eqref{Eq:MuK=varphi} as $\mu_k([u_{(k)}]) = [\varphi_k(u_{(k)})]$. Let $\mu_k^{-1}$ denote the inverse of the rational transduction $\mu_k$. Then the above amounts to saying that 
\begin{equation}\label{Eq:Mu_is_injective}
(\mu_k^{-1} \circ \mu_k)(L) = L
\end{equation}
for every $L \subseteq \sh_k(A)$. 

\begin{lemma}\label{Lem:cR_anc_is_transduction}
Let $L \subseteq \sh_k(A)$. Then $\sR_{(k)}(L) = \mu_k^{-1}\sR^\varphi_{(k)}(\mu_k(L))$.
\end{lemma}
\begin{proof}
By \eqref{Eq:Def_of_k_ancestors}, we have $[w_{(k)}] \in \sR_{(k)}(L)$ if and only if there exists $[u_{(k)}] \in L$ such that $w_i \xra{\sR_i} u_i$ for all $0 \leq i \leq k$, which is true if and only if $\varphi_i(w_i) \xra{\sR^\varphi_i} \varphi_i(u_i)$, i.e. $[\varphi_k(w_{(k)})] \in \sR_{(k)}^\varphi (\varphi_k(u_{(k)}))$. But this is simply saying $\mu_k([w_{(k)}]) \in \sR_{(k)}^\varphi(\mu_k([u_{(k)}]))$, which by \eqref{Eq:Mu_is_injective} is equivalent to
\[
[w_{(k)}] \in \mu_k^{-1}(\sR_{(k)}^\varphi(\mu_k([u_{(k)}]))).
\]
With less cumbersome notation, we have proved that $w \in \sR_{(k)}(L)$ if and only if there is some $u \in L$ such that 
\[
w \in \mu_k^{-1}(\sR_{(k)}^\varphi(\mu_k(u))).
\]
In other words, as $w$ is arbitrary, we have $\sR_{(k)}(L) = \mu_k^{-1}\sR^\varphi_{(k)}(\mu_k(L))$.
\end{proof}

\begin{proof}[Proof of Proposition~\ref{Prop:Polypartisan_preserves}]
As $L \in \cc$, we have $\mu_k(L) \in \cc$, as the super-$\AFL$ $\cc$ is closed under rational transduction. As $\sR_i$ is $\cc$-monadic, so too is $\sR_i^\varphi$ for $0 \leq i \leq k$. As $\mu_k(L) \subseteq \widehat{\sh}_k(A_0, \dots, A_k)$, we conclude by Lemma~\ref{Lem:Weak_polypartisan} that $\sR^\varphi_{(k)}(\mu_k(L))$ is in $\cc$. Finally, as $\mu_k^{-1}$ is a rational transduction, the language $\mu_k^{-1}\sR^\varphi_{(k)}(\mu_k(L))$ is in $\cc$; by Lemma~\ref{Lem:cR_anc_is_transduction}, $\sR_{(k)}(L)$ is hence in $\cc$.
\end{proof}

This completes our discussion of polypartisan ancestors.

\section{Monoid free products}\label{Sec:MFP}

In this section we will consider monoid free products. We begin by proving the main theorem for free products of word-hyperbolic monoids with $1$-uniqueness (Theorem~\ref{Thm:Monoid-fp-theorem-1-unqiueness}). We then present a theorem which applies outside the $1$-uniqueness case, to the cases when the combings $R_i$ of the factor monoids $M_i$ satisfy $R_i^\ast = R_i$ (Theorem~\ref{Thm:star-hyp-fp}). We then argue that these two cases are, in a certain sense, complementary (\S\ref{Subsec:NewDef}). 

\subsection{The case of $1$-uniqueness}\label{Subsec:1-uniqueness}

Suppose that $M_i$ (for $i=1,2$) is a word-hyperbolic monoid with $1$-uniqueness, with respect to the regular combing $R_i$. By definition, the only word in $R_i$ which represents the identity of $M_i$ is $\varepsilon$. Let $R_i' = R_i - \{ \varepsilon\}$. Then it is clear that every alternating word in $\Alt(R_1', R_2')$ is reduced; for if $u_0 u_1 \cdots u_n$ is the alternating factorisation of $u \in \Alt(R_1', R_2')$, and $u$ were not reduced, then $u_i = 1$ in either $M_1$ or $M_2$ for some $0 \leq i \leq n$, and hence $u_i \equiv \varepsilon$, a contradiction to $u_i \in R_1' \cup R_2'$. Hence, by Lemma~\ref{Lem:SGPFP-combing-equality}, monoid free products of monoids with $1$-uniqueness behave essentially as semigroup free products of the same monoids, up to the fact that the product of two reduced sequences may not be reduced. 

Using monadic ancestry, we may deal with this latter issue, and show the following main theorem.

\begin{theoremB}
Let $M_1, M_2$ be two word-hyperbolic monoids with $1$-uniqueness (with uniqueness). Then the monoid free product $M_1 \ast M_2$ is word-hyperbolic with $1$-uniqueness (with uniqueness).
\end{theoremB}
\begin{proof}
Suppose $M_1$ (resp. $M_2$) is word-hyperbolic with $1$-uniqueness with respect to the regular combing $R_1$ (resp. $R_2$). As usual, we let $R = \Alt(R_1, R_2)$. If $M_1$ resp. $M_2$ are word-hyperbolic with $1$-uniqueness, then the only element of $R$ representing the identity element is $\varepsilon$, as the only element of $R_1$ resp. $R_2$ representing the identity element of $M_1$ resp. $M_2$ is $\varepsilon$. Analogously, if $M_1, M_2$ are word-hyperbolic with uniqueness, then every alternating word is reduced, and hence every pair of distinct words in $R$ represent distinct elements of $M$ by Lemma~\ref{Lem:Monoid-reduced-fp-NFT}. Hence it suffices to show that $M$ is word-hyperbolic with respect to $R$. 

For $i=1, 2$, we define the monadic rewriting system 
\begin{equation}
\sS_i = \{ (u \#_1 v, \#_1) \mid u, v \in R_i, u \cdot v = 1 \text{ in $M_i$} \}.
\end{equation}
Now, the language of left-hand sides of $\#_1$ in $\sS_i$ is 
\[
\{ u \#_1 v \mid u, v \in R_i, u \cdot v = 1 \text{ in $M_i$ }\} = \cT_{M_i}(R_i) / \{ \#_2 \varepsilon \},
\]
where $/$ denotes the right quotient, in this case by the regular language $\{ \#_2 \varepsilon \}$. As $\cT_{M_i}(R_i)$ is a context-free language, so too is the quotient of $\cT_{M_i}$ by any regular language. We conclude that $\sS_i$ is a context-free monadic rewriting system. Hence the union $\sS = \sS_1 \cup \sS_2$ is also a context-free monadic system. 

We will define the language 
\begin{equation}\label{Eq:Def_of_cL1}
\cl_1 = \anc{\sS}(\cT_S(R)) \cap R \#_1 R \#_2 R^\trev.
\end{equation}

We will prove that $\cl_1 = \cT_M(R)$, which suffices to prove the theorem (as a quick argument will show). This highlights that the language-theoretic properties of the monoid free product of word-hyperbolic monoids with $1$-uniqueness are not significantly more complicated than those of the semigroup free product of the same. One direction is easy, and depends on little more than the two facts that (i) if $u \cdot v =_S w$, then $u \cdot v =_M w$ for $u, v, w \in R$; and (ii) if $u \cdot v =_{M_i} 1$, then $u \cdot v =_{M} 1$ for $u, v \in R_i$. 

\begin{lemma}\label{Lem:cL1_in_cTM(R)}
$\cl_1 \subseteq \cT_M(R)$.
\end{lemma}
\begin{proof}
The proof of this is entirely analogous to that of Lemma~\ref{Lem:w_anc_is_in_table}, with one minor addition: note that if $w_1 \#_1 w_2 \#_2 w_3^\trev \in \cT_S(R)$, then we have $w_1 \cdot w_2 =_S w_3$ and hence also $w_1 \cdot w_2 =_M w_3$. If $u, v \in R_i$ are such that $u \cdot v =_{M_i} 1$ for some $i=1,2$, then $u \cdot v =_M 1$, so also $w_1u \cdot vw_2 =_M w_3$. Hence, if $w_1 u, vw_2 \in R$, then we conclude that $w_1(u \#_1 v)w_2 \#_2 w_3^\trev \in \cT_M(R)$, and 
\[
w_1(u\#_1v)w_2\#_2w_3^\trev \xr{\sS} w_1 \#_1 w_2 \#_2 w_3^\trev. 
\]
We leave the (simple) details to the reader. 
\end{proof}

We remark (as shall be needed in \S\ref{Subsec:Rast=R}) that the assumption of $1$-uniqueness is not needed to prove Lemma~\ref{Lem:cL1_in_cTM(R)}. The non-trivial part of the equality $\cT_M(R) = \cl_1$ is given by the following lemma. 

\begin{lemma}\label{Lem:cTM(R)_in_cL1}
$\cT_M(R) \subseteq \cl_1$. 
\end{lemma}
\begin{proof}
Suppose $w \equiv w_1 \#_1 w_2 \#_2 x^\trev \in \cT_M(R)$. Then $w_1, w_2, x \in R$, and $w_1 \cdot w_2 =_M x$. By $1$-uniqueness, $w_1, w_2$, and $x$ are all necessarily reduced (though $w_1w_2$ may not be). Hence we can apply Lemma~\ref{Lem:Mult_table_MON_lemma}. If we are in case (1) or (2), then by Lemma~\ref{Lem:Mult_table_SGP_lemma} we have $w_1 \cdot w_2 =_S x$, and so $w_1 \#_1 w_2 \#_2 x^\trev \in \cT_S(R)$, and hence, using no rewritings, we find
\[
w \equiv w_1 \#_1 w_2 \#_2 x^\trev \in \anc{\sS}(\cT_S(R)) \cap R \#_1 R \#_2 R^\trev = \cl_1.
\] 
If we are instead in case (3), then we must use $\sS$ non-trivially. As $x'_k, x_k'' \in R_{X'(k)}$ satisfy $x'_k \cdot x''_k =_{M_{X'(k)}} 1$, we have $(x'_k \#_1 x''_k, \#_1) \in \sS_{X'(k)} \subseteq \sS$. Hence also
\begin{align}\label{Eq:w-rewrites-inductive-case-3}
w \equiv w_1 \#_1 w_2 \#_2 x^\trev \equiv w_1' (x'_k \#_1 x''_k) w_2' \#_2 x^\trev \xr{\sS} w_1' \#_1 w_2' \#_2 x^\trev.
\end{align}
As $w_1', w_2', x \in R$ satisfy $w_1' \cdot w_2' =_M x$, and $|w_1'|+|w_2'| < |w_1|+|w_2|$, we may use induction on the parameter $|w_1| + |w_2|$ (the base cases being cases (1) and (2) above), where the inductive hypothesis yields $w_1' \#_1 w_2' \#_2 x^\trev \in \cl_1$. Thus $w_1' \#_1 w_2' \#_2 x^\trev \in \anc{\sS}(\cT_S(R))$, so by \eqref{Eq:w-rewrites-inductive-case-3} we also have $w \in \anc{\sS}(\cT_S(R))$. We conclude by induction that $w \in \cl_1$, as desired. 
\end{proof}

Hence we have found a regular combing $R$ of $M$ such that $\cT_M(R)$ is given by the right-hand side of \eqref{Eq:Def_of_cL1}. The right-hand side of \eqref{Eq:Def_of_cL1} is context-free, by the following chain of reasoning: (i) $\cT_S(R) \in \CF$ by Theorem~\ref{Thm:Sgp_extendable_result}; and hence (ii) $\anc{\sS}(\cT_S(R)) \in \CF$, as the class of context-free languages has the monadic ancestor property and $\sS$ is a context-free monadic rewriting system; and (iii) thus $\cT_M(R) \in \CF$ as $\CF$ is closed under intersection with regular languages. Hence $(R, \cT_M(R))$ is a word-hyperbolic structure for $M = M_1 \ast M_2$. 
\end{proof}

Word-hyperbolicity with $1$-uniqueness is not an unusual phenomenon. For example, it always holds in hyperbolic groups, so we find the following immediate corollary of Theorem~\ref{Thm:Monoid-fp-theorem-1-unqiueness}. 

\begin{corollary}\label{Cor:groups}
The free product of two hyperbolic groups is hyperbolic. 
\end{corollary}
\begin{proof}

By \cite[Theorem~1]{Gilman2002} (cf. also \cite[Corollary~4.3]{Duncan2004}), a group is hyperbolic (in the geometric sense) if and only if it is word-hyperbolic (in the language-theoretic sense of this paper). Hence, as the monoid free product of two groups is the same as the (ordinary) free product of two groups, in view of Theorem~\ref{Thm:Monoid-fp-theorem-1-unqiueness} it suffices to show that hyperbolic groups are word-hyperbolic with $1$-uniqueness. But every hyperbolic group $G$, generated by a finite set $A$, is word-hyperbolic with respect to the regular combing $R \subseteq A^\ast$ given by the language of geodesics in the Cayley graph of $G$, and there is only one geodesic corresponding to the identity element, cf. \cite[Theorem~4.2]{Coornaert1990}. 
\end{proof}

Of course, Corollary~\ref{Cor:groups} is well-known in geometric group theory, and is not difficult to show geometrically. Our approach, via Theorem~\ref{Thm:Monoid-fp-theorem-1-unqiueness}, gives a proof which instead goes via formal language theory.

\subsection{$\star$-word-hyperbolic monoids}\label{Subsec:Rast=R}

Suppose $M_i$ (for $i=1,2$) is a word-hyperbolic monoid with respect to the regular combing $R_i$. If $R_i^\ast = R_i$, then we say that $M_i$ is $\star$-\textit{word-hyperbolic} (with respect to $R_i$). We do not know if every word-hyperbolic monoid is $\star$-word-hyperbolic, but do not suspect this to be the case: $\star$-word-hyperbolic monoids appear to inch too close to monoids with context-free word problem. 

\begin{example}
It is not difficult to show that the bicyclic monoid $B = \pres{Mon}{b,c}{bc=1}$ is word-hyperbolic with respect to the regular combing $c^\ast b^\ast$ (as is shown explicitly in \cite[Example~3.8]{Duncan2004}). Of course, for this combing, we have $(c^\ast b^\ast)^\ast \neq c^\ast b^\ast$. However, $B$ is also word-hyperbolic with respect to the combing $\{ b, c \}^\ast$, as is easily seen by using the complete monadic rewriting system $(bc, 1)$ (cf. also the first few sentences of \cite[Theorem~3.1]{Cain2012}, coupled with \cite[Corollary~3.8]{Book1982}). In particular the bicyclic monoid is $\star$-word-hyperbolic.

In fact, this example is a consequence of the general fact that the group of units of the bicyclic monoid is trivial. Recall that a monoid is \textit{special} if every defining relation is of the form $w_i = 1$ (see \cite[Chapter~III]{Adian1966}). As proved by the author, a special monoid $M$ has context-free word problem -- in the sense of Duncan \& Gilman \cite[\S5]{Duncan2004} -- if and only if its group of units $U(M)$ is virtually free \cite{NybergBrodda2020b}. Any monoid generated by a finite set $A$ clearly has context-free word problem if and only if it is word-hyperbolic with respect to the regular combing $A^\ast$ (one direction is trivial by a rational transduction; the other is observed at the beginning of the proof of \cite[Theorem~3.1]{Cain2012}). Cf. also \cite[Theorem~2(2)]{Gilman2002}. Thus any context-free monoid is $\star$-word-hyperbolic. 
\end{example}

The main theorem of this section is the following, which uses polypartisan ancestors:

\begin{theorem}\label{Thm:star-hyp-fp}
Let $M_1, M_2$ be two $\star$-word-hyperbolic monoids. Then the monoid free product $M_1 \ast M_2$ is $\star$-word-hyperbolic. 
\end{theorem}
\begin{proof}
Suppose $M_1, M_2$ are $\star$-word-hyperbolic monoids with respect to the regular combings $R_1$ resp. $R_2$. Then $R_1^\ast = R_1$ and $R_2^\ast = R_2$. Let, as usual, $R = \Alt(R_1, R_2)$, and let $M$ denote the monoid free product $M_1 \ast M_2$. However, note that, in this case, we can simplify $\Alt(R_1, R_2) = (R_1 \cup R_2)^\ast$. It suffices to show that $\cT_M(R)$ is a context-free language, as $R$ clearly combs $M$. We have done most of the heavy lifting in the proofs of Theorem~\ref{Thm:Sgp_extendable_result} and Theorem~\ref{Thm:Monoid-fp-theorem-1-unqiueness}. However, unlike in the setting of these theorems, we cannot assume that every element of $R$ is reduced. We remedy this with a context-free monadic rewriting system.

We first define, for $i=1,2$, the rewriting systems
\[
\sT_i = \{ (w, 1) \mid w \in R_i, w =_{M_i} 1 \}.
\]
Then $\sT_i$ is a context-free monadic rewriting system, as the left-hand sides of $1$ are obtained by taking a right quotient of the context-free multiplication table $\cT_{M_i}(R_i)$ by the regular language $\#_1\#_2$. Note that for every rule $(w, 1) \in \sT_i$, we have $w =_M 1$, by the properties of the monoid free product. We let $\sT = \sT_1 \cup \sT_2$, which is also a context-free monadic system. We let further $\sT^\trev$ be the system consisting of all rules $(w^\trev, 1)$ such that $(w, 1) \in \sT$. Then $\sT^\trev$ is a context-free monadic rewriting system, as the class $\CF$ is closed under reversal. 

Note that for every word $w \in R$, there exists some (not necessarily unique) reduced $w' \in R$ such that $w \xra{\sT} w'$. Of course, as $\sT$ is $M$-equivariant, for such $w, w'$ we have $w =_M w'$. 

Let $\sR_1 = \sR_2 = \sT$, and let $\sR_3 = \sT^\trev$. Consider the polypartisan ancestor 
\begin{equation}\label{Eq:Def_of_cl2}
\cl_2 = \sR_{(3)}(\cl_1) \cap R \#_1 R \#_2 R^\trev.
\end{equation}
Recall the definition of $\cl_1$ as \eqref{Eq:Def_of_cL1}, and see \S\ref{Sec:Polypartisan} for notation pertaining to polypartisan ancestors. By Lemma~\ref{Lem:cL1_in_cTM(R)} (and the remark following it), we have $\cl_1 \subseteq \cT_M(R)$. Hence $\cl_2$ consists of some collection of words of the form $w_1 \#_1 w_2 \#_2 x^\trev$ with $w_1, w_2, x \in R$ such that there exist words $w_1', w_2', x' \in R$ with $w_1' \cdot w_2' =_M x'$. As the systems $\sR_1$ and $\sR_2$ are $M$-equivariant, and $\sR_3$ is $M^\trev$-equivariant, it follows easily that $w_1 =_M w_1', w_2 =_M w_2'$, and $x =_M x'$. Thus $w_1 \cdot w_2 =_M x$, so it follows that $\cl_2 \subseteq \cT_M(R)$. We show the reverse inclusion, which (by a simple argument) will suffice to show that $M$ is word-hyperbolic. 

\begin{lemma}\label{Lem:cTM_=_CL2}
$\cT_M(R) = \cl_2$.
\end{lemma}
\begin{proof}
We have shown the inclusion $\cl_2 \subseteq \cT_M(R)$ above. For the inclusion $\cT_M(R) \subseteq \cl_2$, suppose that $w \equiv w_1 \# w_2 \#_2 x^\trev \in \cT_M(R)$. First, $w_1$ is an alternating product, say $w_1 \equiv w_{1,0} w_{1,1} \cdots w_{1,k}$, where $w_{1,i} \in R_{X(i)}$ for some parametrisation $X$. Now, $w_1$ may not be reduced; however, by removing each factor $w_{i,j}$ with $w_{i,j} =_{M_{X(i)}} = 1$, we obtain a reduced word $w_1' \equiv w_{1,i_1} w_{1,i_2} \cdots w_{1,i_\ell}$. Now, it may be the case that $w_{1,i_j} \sim w_{1,i_{j+1}}$, i.e. that $w_{1,i_j}$ and $w_{1,i_{j+1}}$ come from the same factor, and that the factorisation of $w_1'$ is not alternating. However, and crucially, as $R_{X(i)}^\ast = R_{X(i)}$, we can find some word $w''_{1,i_j} \in R_{X(i)}$ such that $w''_{1,i_j} \equiv w'_{1,i_j} w'_{1,i_{j+1}}$. By merging all terms in this way, we find an alternating factorisation of $w_1'$, so $w_1' \in R$. Thus there exists a word $w_1' \in R$ such that $w_1 \xra{\sT} w_1'$. In exactly the same way, there are words $w_2', x' \in R$ such that $w_2 \xra{\sT} w_2'$ and $x \xra{\sT} x'$. In particular, $x^\trev \xra{\sT^\trev} (x')^\trev$. We note in passing that $w_1' \cdot w_2' =_M x'$, by $M$-equivariance. It follows from the above that 
\begin{equation}\label{Eq:Annoying-poly-mon-fp}
w \equiv w_1 \#_1 w_2 \#_2 x^\trev \in \sR_{(3)}(\{ w_1' \#_1 w_2' \#_2 (x')^\trev \}) \cap R \#_1 R \#_2 R^\trev.
\end{equation}

As the words $w_1', w_2', x' \in R$ are reduced and satisfy $w_1' \cdot w_2' =_M x'$, we have 
\begin{equation}\label{Eq:w1'w2'x'_in_cl}
w_1' \#_1 w_2' \#_2 (x')^\trev \in \cl_1.
\end{equation}
From \eqref{Eq:Annoying-poly-mon-fp} and \eqref{Eq:w1'w2'x'_in_cl}, we find immediately by the definition \eqref{Eq:Def_of_cl2} that $w \in \cl_2$, which is what was to be shown. 
\end{proof}

To finish our proof, we must simply conclude that $\cl_2$ is context-free, which follows by combining the facts that (i) $\cl_1$ is a context-free language (the proof of this uses nothing about $1$-uniqueness); (ii) $\sR_{(3)}(\cl_1)$ is a context-free language by Proposition~\ref{Prop:Polypartisan_preserves}; and (iii) the intersection of a context-free language with a regular language is context-free. Hence, as $\cT_M(R) = \cl_2$ by Lemma~\ref{Lem:cTM_=_CL2}, it follows that $(R, \cT_M(R))$ is a word-hyperbolic structure for $M$; as 
\[
R^\ast = \Alt(R_1, R_2)^\ast = ((R_1 \cup R_2)^\ast)^\ast = (R_1 \cup R_2)^\ast = \Alt(R_1, R_2) = R,
\]
it follows that $M$ is $\star$-word-hyperbolic.

\end{proof}

The reader may feel somewhat unsatisfied of the lack of a theorem stating simply that ``the free product of two word-hyperbolic monoids is word-hyperbolic'' (see also \S\ref{Sec:Finalremarks}). However, the combination of Theorem~\ref{Thm:Monoid-fp-theorem-1-unqiueness} and Theorem~\ref{Thm:star-hyp-fp} essentially covers all cases of interest. We demonstrate this now, by showing that the $\star$-word-hyperbolic case can be viewed as a ``complement'' to the $1$-uniqueness case treated in \S\ref{Subsec:1-uniqueness}. 

\begin{proposition}\label{Prop:if_not_1_unique_but_alt_thenRast}
Suppose $M_1, M_2$ are word-hyperbolic without $1$-uniqueness with respect to regular combings $R_1$ resp. $R_2$, and suppose further that the monoid free product $M = M_1 \ast M_2$ is word-hyperbolic with respect to some regular combing $R$. If $\Alt(R_1, R_2) \subseteq R$, then $M_1, M_2$, and $M$ are all $\star$-word-hyperbolic.
\end{proposition}
\begin{proof}
Suppose that $z \in R_1$ is a non-empty word such that $z =_{M_1} 1$. Then also $z =_M 1$. Let 
\[
u_1, \dots, u_k, v_1, \dots, v_m, w_1, \dots, w_n \in R_2
\]
be any words such that
\[
(u_1 u_2 \cdots u_k) \cdot (v_1 v_2 \cdots v_m) =_{M_2} (w_1 \cdots w_n).
\]
Then certainly 
\begin{equation}\label{Eq:zwzwzwz}
z u_1 z u_2 \cdots z u_k \cdot z v_1 z v_2 \cdots z v_m =_M z w_1 z w_2 z \cdots z w_n.
\end{equation}
Now the left-hand side of \eqref{Eq:zwzwzwz} is of the form $r \cdot s$, where $r, s \in \Alt(R_1, R_2)$, and the right-hand side is also an element of $\Alt(R_1, R_2)$. Hence, we have 
\begin{equation}\label{Eq:zhashashash}
z u_1 z u_2 \cdots z u_k \#_1 z v_1 z v_2 \cdots z v_m \#_2 (z w_1 z w_2 z \cdots z w_n)^\trev
\end{equation}
is an element of $\Alt(R_1, R_2) \#_1 \Alt(R_1, R_2) \#_2 \Alt(R_1, R_2)^\trev$. As $\Alt(R_1, R_2) \subseteq R$, we find that \eqref{Eq:zhashashash} is an element of $R \#_1 R \#_2 R^\trev$, and hence in $\cT_M(R)$. 

We can thus simulate the multiplication table for $M_2$ with respect to $R_2^\ast$ by using $\cT_M(R)$, and inserting sufficiently many $z$-symbols between the words in $R_2^\ast$; more rigorously, we can perform a rational transduction of $\cT_M(R)$ to first obtain all words of the form \eqref{Eq:zhashashash}, and then kill all symbols $z$ by a homomorphic image, and in this way obtain $\cT_{M_2}(R_2^\ast)$, which will thus be context-free. Thus $M_2$ is $\star$-word-hyperbolic; by symmetry, so too is $M_1$. By Theorem~\ref{Thm:star-hyp-fp}, so too is $M$. 
\end{proof}

We remark on why this proposition is useful. Suppose the notation of the proposition. Given the ``alternating'' nature of a free product, it is very natural to ask for a regular combing $R$ of $M$ to at least \textit{contain} the alternating products of elements from $R_1$ and $R_2$. Indeed, if it did not, then the regular combing of the free product could be seen as wholly artificial, and not in any way dependent on the structure of the free factors. In this natural setting, Proposition~\ref{Prop:if_not_1_unique_but_alt_thenRast} then tells us: if $M_1$ and $M_2$ are word-hyperbolic, but without $1$-uniqueness, then we must have that $M_1$ and $M_2$ are in fact $\star$-word-hyperbolic. We elaborate on this remark in \S\ref{Subsec:NewDef}, and use this to suggest that a new definition of word-hyperbolic monoid may be suitable. No new results are presented therein, and so may be skipped without losing any readability of \S\ref{Sec:Finalremarks}.

\subsection{$1$-uniqueness as the norm}\label{Subsec:NewDef}

The definition of word-hyperbolic semigroups by Duncan \& Gilman has been noted by Cain \& Maltcev \cite{Cain2016} to lead to some minor technical issues to be fixed. Namely, Cain \& Maltcev note the following: there exist a finite set $A$, a regular language $R \subseteq A^+$ and two non-isomorphic semigroups $S, T$ each generated by $A$ such that $\cT_S(R) = \cT_T(R)$. That is, the word-hyperbolic structure $(R, \cT_S(R))$ does not necessarily determine the semigroup $S$ up to isomorphism.\footnote{On the other hand, if considering monoids, this is not an issue, as the problem arises from the fact that some generators can be indecomposable in a semigroup, which never happens in monoids.} If, on the other hand, the associated homomorphism $\pi \colon A^+ \to S$ is assumed to be injective on $A$, then one can show that uniqueness up to isomorphism \textit{does} hold \cite[Proposition~3.5]{Cain2016}, and that furthermore every word-hyperbolic semigroup admits a word-hyperbolic structure with this additional ``injectivity on generators'' requirement \cite[Proposition~3.6]{Cain2016}. It is therefore no real restriction to impose the requirement on word-hyperbolic semigroups that $\pi$ be injective on the generators. 

In a similar vein, we would like to suggest that for word-hyperbolic \textit{monoids} the earlier result (Proposition~\ref{Prop:if_not_1_unique_but_alt_thenRast}) demonstrates that $1$-uniqueness in word-hyperbolic monoids is natural. This argument is based on three desired premisses: 
\begin{enumerate}
\item the free product of two word-hyperbolic monoids ought to be word-hyperbolic; 
\item a word-hyperbolic structure for a free product should reflect the structure of the free factors in an alternating manner; and 
\item $\star$-word-hyperbolicity should be exceptional, rather than the norm. 
\end{enumerate}

If these premisses are accepted, and (2) is interpreted as in the paragraph following Proposition~\ref{Prop:if_not_1_unique_but_alt_thenRast}, then we conclude from Proposition~\ref{Prop:if_not_1_unique_but_alt_thenRast} that any given word-hyperbolic monoid ought to be either $\star$-word-hyperbolic, or else is word-hyperbolic with $1$-uniqueness. The third premiss would therefore guide us to prescribing that word-hyperbolic monoids with $1$-uniqueness should be the norm. If the premisses are accepted, a natural definition of word-hyperbolic monoid would thus be the following: a monoid $M$ is word-hyperbolic if and only if it admits a finite generating set $A$ and a regular combing $R$ such that (i) the multiplication table $\cT_M(R)$ is context-free; and (ii) $\varepsilon \in R$, and this is the only word in $R$ which represents $1 \in M$. If this were the definition of word-hyperbolic monoid, then the free product of two word-hyperbolic monoids is again word-hyperbolic (Theorem~\ref{Thm:Monoid-fp-theorem-1-unqiueness}). 

Whether these premisses (1)--(3) are acceptable or not depends on the reader. Ideally, we would like to bypass this definition-based argument and say that every word-hyperbolic monoid admits a word-hyperbolic structure with $1$-uniqueness, but we do not know whether this is the case. Indeed, one might suspect that this is not the case, as there are word-hyperbolic monoids which do not admit any word-hyperbolic structure with uniqueness \cite{Cain2012}. 

\section{Super-$\AFL$s and $\cc$-tabled groups}\label{Sec:Finalremarks}

\noindent The observant reader may have noticed that, for all our usage of the properties of context-free languages, we have nowhere used the words ``context-free grammar'' or ``pushdown automaton'', or any of the usual specifications of context-free languages. Indeed, we have only used two properties of the class $\CF$ of context-free languages, namely: 
\begin{enumerate}
\item $\CF$ is a reversal-closed full $\AFL$; and 
\item $\CF$ has the monadic ancestor property (see \S\ref{Subsec:Rewritingsystems}). 
\end{enumerate}

That is, in the terminology of \S\ref{Subsec:ETOL-subst}, we have only used the property that $\CF$ is a reversal-closed super-$\AFL$.\footnote{Similarly general statements involving reversal-closed super-$\AFL$s appear as the main results in previous work by  the author \cite{NybergBrodda2020c, NybergBrodda2020b, NybergBrodda2021f}. } Hence the main results of this article (Theorem~\ref{Thm:Sgp_extendable_result}, Theorem~\ref{Thm:Monoid-fp-theorem-1-unqiueness}, Theorem~\ref{Thm:star-hyp-fp}, and their corollaries) remain valid if $\CF$ is replaced in the definition of word-hyperbolicity by any other reversal-closed super-$\AFL$, such as $\IND$ or $\ETOL$. We have chosen not to state our theorems in this general form to maintain clarity; there does not, at present, seem to be a great deal of interest in the language-theoretic properties of multiplication tables outside the case of $\CF$ (i.e. hyperbolicity). However, some questions of this nature have been discussed in personal communication with the author; furthermore, due to the recent resurgence of interest in the class of $\ETOL$-languages (see \S\ref{Subsec:ETOL-subst}), which forms a super-$\AFL$, we therefore opt to include this discussion in this final section. 

Let $S$ be a semigroup, finitely generated by $A$. We say that a regular combing $R \subseteq A^+$ is a (language of) \textit{normal forms} for $S$ if every element of $S$ is represented by exactly one word in $S$. We extend this in the natural way to monoids and groups. For example, the language $a^\ast b^\ast$ is a language of normal forms for the free commutative monoid $\pres{Mon}{a,b}{ab=ba}$, and the (regular) language of freely reduced words over $(A \cup A^{-1})^\ast$ is a language of normal forms for the free group on $A$.

\begin{definition}
Let $S$ be a semigroup, finitely generated by $A$. Let $\cc$ be a class of languages. We say that $S$ is $\cc$-\textit{tabled} if there exists a regular language $R \subseteq A^+$ of normal forms for $S$ such that the multiplication table $\cT_S(R)$ lies in $\cc$. 
\end{definition}

For example, the condition of being $\CF$-tabled is the same as being word-hyperbolic \textit{with uniqueness}. The definition is extended in the obvious way to monoids and groups. Recently Duncan, Evetts, Holt \& Rees (private communication) have proved that the free product of two $\EDTOL$-tabled groups is again $\EDTOL$-tabled. We are in a place to complement these results since, as justified earlier, the statements of Theorem~\ref{Thm:Sgp_extendable_result}, Theorem~\ref{Thm:Monoid-fp-theorem-1-unqiueness}, Theorem~\ref{Thm:star-hyp-fp}, and their corollaries, can be altered to replace ``word-hyperbolic'' with ``$\cc$-tabled'' for any reversal-closed super-$\AFL$ $\cc$ without any loss of validity. Furthermore, the uniqueness of representatives in a normal form allow us to bypass the technical conditions of e.g. $1$-uniqueness. 

\begin{theoremC}\label{Thm:A'}
Let $\cc$ be a reversal-closed super-$\AFL$. Let $S_1, S_2$ be $\cc$-tabled semigroups. Then the semigroup free product $S_1 \ast S_2$ is $\cc$-tabled.
\end{theoremC}

\begin{theoremC}\label{Thm:B'}
Let $\cc$ be a reversal-closed super-$\AFL$. Let $M_1, M_2$ be $\cc$-tabled monoids. \\ Then the monoid free product $M_1 \ast M_2$ is $\cc$-tabled.
\end{theoremC}

These theorems, which are quite elegant to state, demonstrate that the property of having unique normal forms ensures that free products behave very well, although much of the difficulty involving words representing the identity substituted into other words is bypassed in this way. Additionally, Theorem~\ref{Thm:B'} yields the corresponding result for groups and group free products, too, as the monoid free product of two groups coincides with the group free product of the same groups. In particular, we find the following corollaries, both corresponding to Corollary~\ref{Cor:groups}:

\begin{corollary}\label{Cor:ETOLtabledgroups}
The free product of two $\ETOL$-tabled groups is $\ETOL$-tabled.
\end{corollary}

\begin{corollary}\label{Cor:INDtabledgroups}
The free product of two $\IND$-tabled groups is $\IND$-tabled.
\end{corollary}

This complements the aforementioned result by Duncan, Evetts, Holt \& Rees for $\EDTOL$-tabled groups. 

\clearpage
{
\bibliography{hyp_fp.bib} 

\begin{thebibliography}{100}

\bibitem{Adian1966}
S.~I. Adian.
\newblock Defining relations and algorithmic problems for groups and
  semigroups.
\newblock {\em Trudy Mat. Inst. Steklov.}, 85:123, 1966.

\bibitem{Aho1967}
Alfred~V. Aho.
\newblock {\em Indexed grammars--an extension of context-free grammars}.
\newblock PhD thesis, Princeton University, 1967.

\bibitem{Aho1968}
Alfred~V. Aho.
\newblock Indexed grammars--an extension of context-free grammars.
\newblock {\em J. Assoc. Comput. Mach.}, 15:647--671, 1968.

\bibitem{Anisimov1969}
A.~V. An\={\i}s\={\i}mov.
\newblock The automorphism group of connected automata.
\newblock In {\em Automata theory ({P}roc. {S}em., {K}iev, 1969), {N}o. 5
  ({R}ussian)}, pages 27--36, 1969.

\bibitem{Anisimov1971}
A.~V. An\={\i}s\={\i}mov.
\newblock Group languages.
\newblock {\em Kibernetika (Kiev)}, (4):18--24, 1971.

\bibitem{Anisimov1972}
A.~V. An\={\i}s\={\i}mov.
\newblock Certain algorithmic questions for groups and context-free languages.
\newblock {\em Kibernetika (Kiev)}, (2):4--11, 1972.

\bibitem{Anisimov1972b}
A.~V. An\={\i}s\={\i}mov.
\newblock Dyck languages of rank 1.
\newblock {\em Dopov\={\i}d\={\i} Akad. Nauk Ukra\"{\i}n. RSR Ser. A}, pages
  483--485, 571, 1972.

\bibitem{Anisimov1973}
A.~V. An\={\i}s\={\i}mov.
\newblock Certain algorithmic questions for a {D}yck language.
\newblock {\em Dokl. Akad. Nauk SSSR}, 211:513--515, 1973.

\bibitem{Anisimov1975b}
A.~V. An\={\i}s\={\i}mov.
\newblock Languages over free groups.
\newblock In {\em Mathematical foundations of computer science ({F}ourth
  {S}ympos., {M}ari\'{a}nsk\'{e} {L}\'{a}zn\v{e}, 1975)}, pages 167--171.
  Lecture Notes in Comput. Sci., Vol. 32. 1975.

\bibitem{Anisimov1975a}
A.~V. An\={\i}s\={\i}mov and Franz~D. Seifert.
\newblock Zur algebraischen {C}harakteristik der durch kontext-freie {S}prachen
  definierten {G}ruppen.
\newblock {\em Elektron. Informationsverarb. Kybernet.}, 11(10-12):695--702,
  1975.

\bibitem{Asveld1977b}
Peter R.~J. Asveld.
\newblock Controlled iteration grammars and full hyper-{AFL}'s.
\newblock {\em Information and Control}, 34(3):248--269, 1977.

\bibitem{Asveld1987}
Peter R.~J. Asveld.
\newblock Complexity aspects of iterated rewriting---a survey.
\newblock In {\em Essays on concepts, formalisms, and tools}, volume~42 of {\em
  CWI Tract}, pages 89--105. Math. Centrum, Centrum Wisk. Inform., Amsterdam,
  1987.

\bibitem{Asveld2001}
Peter R.~J. Asveld.
\newblock An infinite sequence of full {AFL}-structures, each of which
  possesses an infinite hierarchy.
\newblock In {\em Where mathematics, computer science, linguistics and biology
  meet}, pages 175--186. Kluwer Acad. Publ., Dordrecht, 2001.

\bibitem{Asveld2003}
Peter R.~J. Asveld.
\newblock Algebraic aspects of families of fuzzy languages.
\newblock volume 293, pages 417--445. 2003.
\newblock Algebraic methods in language processing (Iowa City, IA, 2000).

\bibitem{Asveld1977}
P.R.J. Asveld.
\newblock Extensions of language families and canonical forms for full
  afl-structures.
\newblock 1977.
\newblock Research supported by Netherlands Organization for the Advancement of
  Pure Research (ZWO).

\bibitem{Asveld1975}
P.R.J. Asveld and J.~{van Leeuwen}.
\newblock Infinite chains of hyper-afl's.
\newblock 1975.
\newblock Research partly supported by Netherlands Organization for the
  Advancement of Pure Research (ZWO).

\bibitem{Awang2017}
J.~Awang, M.~Pfeiffer, and N.~Ru\v{s}kuc.
\newblock Finite presentability and isomorphism of {C}ayley graphs of monoids.
\newblock {\em Proc. Amer. Math. Soc.}, 145(11):4585--4593, 2017.

\bibitem{Beaudry2002}
M.~Beaudry, M.~Holzer, G.~Niemann, and F.~Otto.
\newblock On the relationship between the {M}c{N}aughton families of languages
  and the {C}homsky hierarchy.
\newblock In {\em Developments in language theory ({V}ienna, 2001)}, volume
  2295 of {\em Lecture Notes in Comput. Sci.}, pages 340--348. Springer,
  Berlin, 2002.

\bibitem{Beaudry2003}
M.~Beaudry, M.~Holzer, G.~Niemann, and F.~Otto.
\newblock Mc{N}aughton families of languages.
\newblock {\em Theoret. Comput. Sci.}, 290(3):1581--1628, 2003.

\bibitem{Beauquier1981}
Joffroy Beauquier.
\newblock A remark about a substitution property.
\newblock {\em Math. Systems Theory}, 14(2):189--191, 1981.

\bibitem{Berstel1979}
Jean Berstel.
\newblock {\em Transductions and context-free languages}, volume~38 of {\em
  Leitf\"{a}den der Angewandten Mathematik und Mechanik [Guides to Applied
  Mathematics and Mechanics]}.
\newblock B. G. Teubner, Stuttgart, 1979.

\bibitem{Book1982}
Ronald~V. Book, Matthias Jantzen, and Celia Wrathall.
\newblock Monadic {T}hue systems.
\newblock {\em Theoret. Comput. Sci.}, 19(3):231--251, 1982.

\bibitem{Book1993}
Ronald~V. Book and Friedrich Otto.
\newblock {\em String-rewriting systems}.
\newblock Texts and Monographs in Computer Science. Springer-Verlag, New York,
  1993.

\bibitem{Bridson1996}
Martin~R. Bridson and Robert~H. Gilman.
\newblock Formal language theory and the geometry of {$3$}-manifolds.
\newblock {\em Comment. Math. Helv.}, 71(4):525--555, 1996.

\bibitem{Brough2016}
Tara Brough, Laura Ciobanu, Murray Elder, and Georg Zetzsche.
\newblock Permutations of context-free, {ET}0{L} and indexed languages.
\newblock {\em Discrete Math. Theor. Comput. Sci.}, 17(3):167--178, 2016.

\bibitem{Cain2012}
Alan~J. Cain and Victor Maltcev.
\newblock Context-free rewriting systems and word-hyperbolic structures with
  uniqueness.
\newblock {\em Internat. J. Algebra Comput.}, 22(7):1250061, 14, 2012.

\bibitem{Cain2016}
Alan~J. Cain and Markus Pfeiffer.
\newblock Decision problems for word-hyperbolic semigroups.
\newblock {\em J. Algebra}, 465:287--321, 2016.

\bibitem{Campbell1995}
C.~M. Campbell, E.~F. Robertson, N.~Ru\v{s}kuc, and R.~M. Thomas.
\newblock Semigroup and group presentations.
\newblock {\em Bull. London Math. Soc.}, 27(1):46--50, 1995.

\bibitem{Christensen1974}
P.~A. Christensen.
\newblock Hyper-{AFL}'s and {ET}0{L} systems.
\newblock In {\em {$L$} systems ({T}hird {O}pen {H}ouse, {C}omput. {S}ci.
  {D}ept., {A}arhus {U}niv., {A}arhus, 1974)}, pages 254--257, 327--338.
  Lecture Notes in Comput. Sci., Vol. 15. 1974.

\bibitem{Ciobanu2016}
Laura Ciobanu, Volker Diekert, and Murray Elder.
\newblock Solution sets for equations over free groups are {EDT}0{L} languages.
\newblock {\em Internat. J. Algebra Comput.}, 26(5):843--886, 2016.

\bibitem{Ciobanu2019}
Laura Ciobanu and Murray Elder.
\newblock Solutions sets to systems of equations in hyperbolic groups are
  {EDT}0{L} in {PSPACE}.
\newblock In {\em 46th {I}nternational {C}olloquium on {A}utomata, {L}anguages,
  and {P}rogramming}, volume 132 of {\em LIPIcs. Leibniz Int. Proc. Inform.},
  pages Art. No. 110, 15. Schloss Dagstuhl. Leibniz-Zent. Inform., Wadern,
  2019.

\bibitem{Ciobanu2021}
Laura Ciobanu and Murray Elder.
\newblock The complexity of solution sets to equations in hyperbolic groups.
\newblock {\em Israel J. Math.}, 245(2):869--920, 2021.

\bibitem{Ciobanu2018}
Laura Ciobanu, Murray Elder, and Michal Ferov.
\newblock Applications of {L} systems to group theory.
\newblock {\em Internat. J. Algebra Comput.}, 28(2):309--329, 2018.

\bibitem{Coornaert1990}
M.~Coornaert, T.~Delzant, and A.~Papadopoulos.
\newblock {\em G\'{e}om\'{e}trie et th\'{e}orie des groupes}, volume 1441 of
  {\em Lecture Notes in Mathematics}.
\newblock Springer-Verlag, Berlin, 1990.
\newblock Les groupes hyperboliques de Gromov. [Gromov hyperbolic groups], With
  an English summary.

\bibitem{Culik1975}
Karel~II {\v{C}}ulik and J.~Opatrny.
\newblock Macro {OL}-systems.
\newblock {\em Int. J. Comput. Math.}, 4:327--342, 1975.

\bibitem{Dahmani2011}
Fran\c{c}ois Dahmani and Vincent Guirardel.
\newblock The isomorphism problem for all hyperbolic groups.
\newblock {\em Geom. Funct. Anal.}, 21(2):223--300, 2011.

\bibitem{Diekert2020}
Volker Diekert and Murray Elder.
\newblock Solutions to twisted word equations and equations in virtually free
  groups.
\newblock {\em Internat. J. Algebra Comput.}, 30(4):731--819, 2020.

\bibitem{Downey1974}
P.~J. Downey.
\newblock {\em Formal languages and recursion schemes}.
\newblock PhD thesis, Harvard University, 1974.

\bibitem{Drewes2004}
Frank Drewes and Joost Engelfriet.
\newblock Branching synchronization grammars with nested tables.
\newblock {\em J. Comput. System Sci.}, 68(3):611--656, 2004.

\bibitem{Duncan2004}
Andrew Duncan and Robert~H. Gilman.
\newblock Word hyperbolic semigroups.
\newblock {\em Math. Proc. Cambridge Philos. Soc.}, 136(3):513--524, 2004.

\bibitem{Dunwoody1985}
M.~J. Dunwoody.
\newblock The accessibility of finitely presented groups.
\newblock {\em Invent. Math.}, 81(3):449--457, 1985.

\bibitem{Ehrenfeucht1975}
A.~Ehrenfeucht, G.~Rozenberg, and S.~Skyum.
\newblock A relationship between {ETOL} and {EDTOL} languages.
\newblock {\em Theoret. Comput. Sci.}, 1(4):325--330, 1975/76.

\bibitem{Engelfriet1976}
Joost Engelfriet.
\newblock Surface tree languages and parallel derivation trees.
\newblock {\em Theoret. Comput. Sci.}, 2(1):9--27, 1976.

\bibitem{Engelfriet1981}
Joost Engelfriet.
\newblock Three hierarchies of transducers.
\newblock {\em Math. Systems Theory}, 15(2):95--125, 1981/82.

\bibitem{Engelfriet1985}
Joost Engelfriet.
\newblock Hierarchies of hyper-{AFL}s.
\newblock {\em J. Comput. System Sci.}, 30(1):86--115, 1985.

\bibitem{Epstein1992}
David B.~A. Epstein, James~W. Cannon, Derek~F. Holt, Silvio V.~F. Levy,
  Michael~S. Paterson, and William~P. Thurston.
\newblock {\em Word processing in groups}.
\newblock Jones and Bartlett Publishers, Boston, MA, 1992.

\bibitem{Evetts2022}
A.~Evetts and Levine A.
\newblock Equations in virtually abelian groups: Languages and growth.
\newblock To appear in International Journal of Algebra and Computation, 2022.

\bibitem{Ferte2014}
Julien Fert{\'{e}}, Nathalie Marin, and G\'{e}raud S\'{e}nizergues.
\newblock Word-mappings of level 2.
\newblock {\em Theory Comput. Syst.}, 54(1):111--148, 2014.

\bibitem{Fountain2004}
John Fountain and Mark Kambites.
\newblock Hyperbolic groups and completely simple semigroups.
\newblock In {\em Semigroups and languages}, pages 106--132. World Sci. Publ.,
  River Edge, NJ, 2004.

\bibitem{Garreta2021}
Albert Garreta and Robert~D. Gray.
\newblock On equations and first-order theory of one-relator monoids.
\newblock {\em Inform. and Comput.}, 281:Paper No. 104745, 19, 2021.

\bibitem{Gilman1996}
Robert~H. Gilman.
\newblock A shrinking lemma for indexed languages.
\newblock {\em Theoret. Comput. Sci.}, 163(1-2):277--281, 1996.

\bibitem{Gilman2002}
Robert~H. Gilman.
\newblock On the definition of word hyperbolic groups.
\newblock {\em Math. Z.}, 242(3):529--541, 2002.

\bibitem{Ginsburg1967}
Seymour Ginsburg and Sheila Greibach.
\newblock Abstract families of languages.
\newblock In {\em 8th Annual Symposium on Switching and Automata Theory (SWAT
  1967)}, pages 128--139, 1967.

\bibitem{Ginsburg1969a}
Seymour Ginsburg and Sheila Greibach.
\newblock Abstract families of languages.
\newblock In {\em Studies in abstract families of languages}, pages 1--32. Mem.
  Amer. Math. Soc., No. 87. 1969.

\bibitem{Ginsburg1970b}
Seymour Ginsburg and Sheila Greibach.
\newblock Principal {${\operatorname{AFL}}$}.
\newblock {\em J. Comput. System Sci.}, 4:308--338, 1970.

\bibitem{Ginsburg1970}
Seymour Ginsburg and Edwin~H. Spanier.
\newblock Substitution in families of languages.
\newblock {\em Information Sci.}, 2:83--110, 1970.

\bibitem{Gray2011}
Robert Gray and Mark Kambites.
\newblock A \v{S}varc-{M}ilnor lemma for monoids acting by isometric
  embeddings.
\newblock {\em Internat. J. Algebra Comput.}, 21(7):1135--1147, 2011.

\bibitem{Gray2013}
Robert Gray and Mark Kambites.
\newblock Groups acting on semimetric spaces and quasi-isometries of monoids.
\newblock {\em Trans. Amer. Math. Soc.}, 365(2):555--578, 2013.

\bibitem{Gray2014}
Robert~D. Gray and Mark Kambites.
\newblock A strong geometric hyperbolicity property for directed graphs and
  monoids.
\newblock {\em J. Algebra}, 420:373--401, 2014.

\bibitem{Gray2017}
Robert~D. Gray and Mark Kambites.
\newblock Amenability and geometry of semigroups.
\newblock {\em Trans. Amer. Math. Soc.}, 369(11):8087--8103, 2017.

\bibitem{Gray2020}
Robert~D. Gray and Mark Kambites.
\newblock On cogrowth, amenability, and the spectral radius of a random walk on
  a semigroup.
\newblock {\em Int. Math. Res. Not. IMRN}, (12):3753--3793, 2020.

\bibitem{Greibach1970b}
S.~A. Greibach.
\newblock Chains of full {AFL}'s.
\newblock {\em Math. Systems Theory}, 4:231--242, 1970.

\bibitem{Greibach1969b}
Sheila~A. Greibach.
\newblock Full afls and nested iterated substitution.
\newblock In {\em 10th Annual Symposium on Switching and Automata Theory (swat
  1969)}, pages 222--230, 1969.

\bibitem{Greibach1970}
Sheila~A. Greibach.
\newblock Full {${\rm AFLs}$} and nested iterated substitution.
\newblock {\em Information and Control}, 16:7--35, 1970.

\bibitem{Gromov1987}
M.~Gromov.
\newblock Hyperbolic groups.
\newblock In {\em Essays in group theory}, volume~8 of {\em Math. Sci. Res.
  Inst. Publ.}, pages 75--263. Springer, New York, 1987.

\bibitem{Grunschlag1999}
Z.~Grunschlag.
\newblock {\em Algorithms in Geometric Group Theory}.
\newblock PhD thesis, University of California at Berkley, 1999.

\bibitem{Harrison1978}
Michael~A. Harrison.
\newblock {\em Introduction to formal language theory}.
\newblock Addison-Wesley Publishing Co., Reading, Mass., 1978.

\bibitem{Hayashi1973}
Takeshi Hayashi.
\newblock On derivation trees of indexed grammars: an extension of the
  {$uvwxy$}-theorem.
\newblock {\em Publ. Res. Inst. Math. Sci.}, 9:61--92, 1973/74.

\bibitem{Herman1975}
Gabor~T. Herman and Grzegorz Rozenberg.
\newblock {\em Developmental systems and languages}.
\newblock North-Holland Publishing Co., Amsterdam-Oxford; American Elsevier
  Publishing Co., Inc., New York, 1975.
\newblock With a contribution by Aristid Lindenmayer.

\bibitem{Hoffmann2002}
Michael Hoffmann, Dietrich Kuske, Friedrich Otto, and Richard~M. Thomas.
\newblock Some relatives of automatic and hyperbolic groups.
\newblock In {\em Semigroups, algorithms, automata and languages ({C}oimbra,
  2001)}, pages 379--406. World Sci. Publ., River Edge, NJ, 2002.

\bibitem{Hoffmann2006}
Michael Hoffmann and Richard~M. Thomas.
\newblock A geometric characterization of automatic semigroups.
\newblock {\em Theoret. Comput. Sci.}, 369(1-3):300--313, 2006.

\bibitem{Hoffmann2010}
Michael Hoffmann and Richard~M. Thomas.
\newblock Notions of hyperbolicity in monoids.
\newblock {\em Theoret. Comput. Sci.}, 411(4-5):799--811, 2010.

\bibitem{Hopcroft1979}
John~E. Hopcroft and Jeffrey~D. Ullman.
\newblock {\em Introduction to automata theory, languages, and computation}.
\newblock Addison-Wesley Series in Computer Science. Addison-Wesley Publishing
  Co., Reading, Mass., 1979.

\bibitem{Howie1995}
John~M. Howie.
\newblock {\em Fundamentals of semigroup theory}, volume~12 of {\em London
  Mathematical Society Monographs. New Series}.
\newblock The Clarendon Press, Oxford University Press, New York, 1995.
\newblock Oxford Science Publications.

\bibitem{Istrate1997}
Gabriel Istrate.
\newblock The strong equivalence of {ET}0{L} grammars.
\newblock {\em Inform. Process. Lett.}, 62(4):171--176, 1997.

\bibitem{Jantzen1988}
Matthias Jantzen.
\newblock {\em Confluent string rewriting}, volume~14 of {\em EATCS Monographs
  on Theoretical Computer Science}.
\newblock Springer-Verlag, Berlin, 1988.

\bibitem{Jez2022}
Artur Je{\.{z}}.
\newblock Word equations in non-deterministic linear space.
\newblock {\em J. Comput. System Sci.}, 123:122--142, 2022.

\bibitem{Klein1979}
Felix Klein.
\newblock {\em Vorlesungen \"{u}ber die {E}ntwicklung der {M}athematik im 19.
  {J}ahrhundert}.
\newblock Springer-Verlag, Berlin-New York, 1979.
\newblock Two volumes reprinted as one, With a foreword by R. Courant and O.
  Neugebauer.

\bibitem{Kral1970}
Jaroslav Kr{\'{a}}l.
\newblock A modification of a substitution theorem and some necessary and
  sufficient conditions for sets to be context-free.
\newblock {\em Math. Systems Theory}, 4:129--139, 1970.

\bibitem{Kuske2006}
Dietrich Kuske and Markus Lohrey.
\newblock Logical aspects of {C}ayley-graphs: the monoid case.
\newblock {\em Internat. J. Algebra Comput.}, 16(2):307--340, 2006.

\bibitem{Latteux1977}
Michel Latteux.
\newblock "{EDTOL}-syst\`emes ultralin\'eaires et op\'erateurs associ\'es''.
\newblock {\em Laboratoire de Calcul, Univ. Lille I, Villeneuve d'Ascq},
  Publication No. 100, 1977.

\bibitem{Latteux1979}
Michel Latteux.
\newblock Substitutions dans les {EDTOL} syst\`emes ultralin\'{e}aires.
\newblock {\em Inform. and Control}, 42(2):194--260, 1979.

\bibitem{Levine2021}
Alex Levine.
\newblock $\operatorname{EDT0L}$ solutions to equations in group extensions.
\newblock {\em Pre-print}, 2021.
\newblock Available online at arXiv:2108.09390.

\bibitem{Levine2022}
Alex Levine.
\newblock Formal languages, quadratic {D}iophantine equations and the
  {H}eisenberg group.
\newblock {\em Pre-print}, 2022.
\newblock Available online at arXiv:2203.04849.

\bibitem{Lewis1970}
David~J. Lewis.
\newblock Closure of families of languages under substitution operators.
\newblock In Patrick~C. Fischer, Robert Fabian, Jeffrey~D. Ullman, and
  Richard~M. Karp, editors, {\em Proceedings of the 2nd Annual {ACM} Symposium
  on Theory of Computing, May 4-6, 1970, Northampton, Massachusetts, {USA}},
  pages 100--108. {ACM}, 1970.

\bibitem{Lindenmayer1968a}
Aristid Lindenmayer.
\newblock Mathematical models for cellular interactions in development {I.
  F}ilaments with one-sided inputs.
\newblock {\em Journal of Theoretical Biology}, 18(3):280--299, 1968.

\bibitem{Lindenmayer1968b}
Aristid Lindenmayer.
\newblock Mathematical models for cellular interactions in development {II.}
  {S}imple and branching filaments with two-sided inputs.
\newblock {\em Journal of Theoretical Biology}, 18(3):300--315, 1968.

\bibitem{Lyndon1977}
Roger~C. Lyndon and Paul~E. Schupp.
\newblock {\em Combinatorial group theory}.
\newblock Ergebnisse der Mathematik und ihrer Grenzgebiete, Band 89.
  Springer-Verlag, Berlin-New York, 1977.

\bibitem{Magnus1966}
Wilhelm Magnus, Abraham Karrass, and Donald Solitar.
\newblock {\em Combinatorial group theory: {P}resentations of groups in terms
  of generators and relations}.
\newblock Interscience Publishers [John Wiley \& Sons], New York-London-Sydney,
  1966.

\bibitem{McNaughton1988}
Robert McNaughton, Paliath Narendran, and Friedrich Otto.
\newblock Church-{R}osser {T}hue systems and formal languages.
\newblock {\em J. Assoc. Comput. Mach.}, 35(2):324--344, 1988.

\bibitem{Meduna2003}
Alexander Meduna and Martin \v{S}vec.
\newblock Forbidding {ET}0{L} grammars.
\newblock {\em Theoret. Comput. Sci.}, 306(1-3):449--469, 2003.

\bibitem{Muller1983}
David~E. Muller and Paul~E. Schupp.
\newblock Groups, the theory of ends, and context-free languages.
\newblock {\em J. Comput. System Sci.}, 26(3):295--310, 1983.

\bibitem{Neumann1967}
B.~H. Neumann.
\newblock Some remarks on semigroup presentations.
\newblock {\em Canadian J. Math.}, 19:1018--1026, 1967.

\bibitem{Neumann2011}
Peter~M. Neumann.
\newblock {\em The mathematical writings of \'{E}variste {G}alois}.
\newblock Heritage of European Mathematics. European Mathematical Society
  (EMS), Z\"{u}rich, 2011.

\bibitem{NybergBrodda2020c}
C.-F. Nyberg-Brodda.
\newblock On the word problem for compressible monoids.
\newblock {\em Pre-print}, 2020.
\newblock Available online at arXiv:2012.01402.

\bibitem{NybergBrodda2020b}
C.-F. Nyberg-Brodda.
\newblock On the word problem for special monoids.
\newblock {\em Pre-print}, 2020.
\newblock Available online at arXiv:2011.09466.

\bibitem{NybergBrodda2021f}
C.-F. Nyberg-Brodda.
\newblock On the word problem for free products of semigroups and monoids.
\newblock {\em Pre-print (submitted)}, 2021.
\newblock Available at arXiv:2112.10665.

\bibitem{Rabkin2012}
Max Rabkin.
\newblock Ogden's lemma for {ET}0{L} languages.
\newblock In {\em Language and automata theory and applications}, volume 7183
  of {\em Lecture Notes in Comput. Sci.}, pages 458--467. Springer, Heidelberg,
  2012.

\bibitem{Rozenberg1977}
G.~Rozenberg, M.~Penttonen, and A.~Salomaa.
\newblock Bibliography of {L} systems.
\newblock {\em Theoret. Comput. Sci.}, 5(3):339--354, 1977/78.

\bibitem{Rozenberg1976}
G.~Rozenberg and A.~Salomaa.
\newblock The mathematical theory of {L} systems.
\newblock In {\em Advances in information systems science, {V}ol. 6}, pages
  161--206. 1976.

\bibitem{Rozenberg1997}
G.~Rozenberg and A.~Salomaa, editors.
\newblock {\em Handbook of formal languages. {V}ol. 1}.
\newblock Springer-Verlag, Berlin, 1997.
\newblock Word, language, grammar.

\bibitem{Rozenberg1979}
G.~Rozenberg and D.~Vermeir.
\newblock On recursion in {ETOL} systems.
\newblock {\em J. Comput. System Sci.}, 19(2):179--196, 1979.

\bibitem{Rozenberg1974}
Grzegorz Rozenberg and Arto Salomaa, editors.
\newblock {\em L systems}.
\newblock Lecture Notes in Computer Science, Vol. 15. Springer-Verlag,
  Berlin-New York, 1974.
\newblock Including papers presented at the Third Open House arranged by the
  Computer Science Department, University of Aarhus, Aarhus, on January 14--25,
  1974.

\bibitem{Salomaa1973}
Arto Salomaa.
\newblock Macros, iterated substitution, and {L}indenmayer {AFL}'s.
\newblock In {\em {DAIMI PB-18, Dept. Comput. Sci., Univ. Aarhus, Aarhus,
  1973}}. 1973.

\bibitem{Salomaa1974}
Arto Salomaa.
\newblock Iteration grammars and {L}indenmayer {AFL}'s.
\newblock In {\em L systems ({T}hird {O}pen {H}ouse, {C}omput. {S}ci. {D}ept.,
  {A}arhus {U}niv., {A}arhus, 1974)}, pages 250--253, 327--338. Lecture Notes
  in Computer Science, Vol. 15. 1974.

\bibitem{Schwarz1955}
Albert~S. Schwarz.
\newblock The volume invariant of coverings.
\newblock {\em Dokl. Akad. Nauk SSSR (N.S.)}, 105:32--34, 1955.
\newblock The author of the present article has produced an English translation
  of this article, which is available upon request.

\bibitem{Silva2004}
Pedro~V. Silva and Benjamin Steinberg.
\newblock A geometric characterization of automatic monoids.
\newblock {\em Q. J. Math.}, 55(3):333--356, 2004.

\bibitem{Smith2017}
Tim Smith.
\newblock A new pumping lemma for indexed languages, with an application to
  infinite words.
\newblock {\em Inform. and Comput.}, 252:176--186, 2017.

\bibitem{Wussing1969}
Hans Wussing.
\newblock {\em Die {G}enesis des abstrakten {G}ruppenbegriffes. {E}in {B}eitrag
  zur {E}ntstehungsgeschichte der abstrakten {G}ruppentheorie}.
\newblock VEB Deutscher Verlag der Wissenschaften, Berlin, 1969.

\end{thebibliography}
\bibliographystyle{plain}
}
 \end{document}